\documentclass[12pt, oneside, psamsfonts]{amsart}

\newif\ifPDF
\ifx\pdfoutput\undefined\PDFfalse
\else \ifnum \pdfoutput > 0 \PDFtrue
        \else \PDFfalse
        \fi
\fi

\usepackage[centertags]{amsmath}
\usepackage{amsfonts}
\usepackage{mathrsfs}
\usepackage{textcomp}
\usepackage{amssymb}
\usepackage{amsthm}
\usepackage{newlfont}
\usepackage[all]{xy}

\ifPDF
  \usepackage[pdftex]{color, graphicx}
  \usepackage[pdftex, bookmarks, colorlinks]{hyperref}
  \hypersetup{colorlinks=false}

\else
  \usepackage{color}
  \usepackage[dvips]{graphicx}
  \usepackage[dvips]{hyperref}
\fi

\usepackage[scale=0.8]{geometry}

\newtheorem{thm}{Theorem}[section]
\newtheorem{cor}[thm]{Corollary}
\newtheorem{lem}[thm]{Lemma}
\newtheorem{prop}[thm]{Proposition}
\theoremstyle{definition}
\newtheorem{defn}[thm]{Definition}
\theoremstyle{remark}
\newtheorem{rem}[thm]{Remark}

\numberwithin{equation}{section}

\newcommand{\norm}[1]{\left\Vert#1\right\Vert}
\newcommand{\abs}[1]{\left\vert#1\right\vert}

\newcommand{\Real}{\mathbb R}
\newcommand{\Int}{\mathbb Z}
\newcommand{\Comp}{\mathbb C}

\newcommand{\eps}{\varepsilon}

\newcommand{\Kzero}{\textrm{K}_0}

\newcommand{\tr}{\mathrm{T}}

\begin{document}

\title[Strict comparison in uniform Roe algebra]{Strict comparison holds in the uniform Roe algebra of a discrete amenable group}

\author{George A. Elliott}
\address{Department of Mathematics, University of Toronto, Toronto, Ontario, Canada~\ M5S 2E4}
\email{elliott@math.toronto.edu}

\author{Chun Guang Li}
\address{
School of Mathematics and Statistics, Northeast Normal University, Changchun, 130024, P. R.~China
}
\email{licg864@nenu.edu.cn}

\author{Zhuang Niu}
\address{Department of Mathematics and Statistics, University of Wyoming, Laramie, Wyoming, USA, 82071}
\email{zhuangniu@icloud.com}

\author{Jianguo Zhang}
\address{
School of Mathematics and Statistics, Shaanxi Normal University, Xi'an, Shaanxi, P. R.~China
}
\email{jgzhang@snnu.edu.cn}

\thanks{G.A.E.~ is supported by an NSERC Discovery Grant. Z.N.~is supported by a Simons Foundation Grant (MP-TSM-00002606). J.Z.~is supported by NSFC grants (Nos.~12271165, 12301154). Part of the work  was carried out when Z.N.~visited the Fields Institute during the spring break of 2026. He thanks the Fields Institute and FOCUS Program for the support.}
\keywords{Uniform Roe algebra, amenable group, universal minimal set, strict comparison}
\date{\today}
%\dedicatory{}
%\commby{}

%------------------------------------------abstract--------------------------------------
\begin{abstract}
Let $\Gamma$ be a countable discrete amenable group, and let $A=l^\infty(\Gamma) \rtimes \Gamma$. It is shown that if $a, b \in A \otimes \mathcal K$ are positive elements such that $$\mathrm{d}_\tau(a) < \mathrm{d}_\tau(b),\quad \tau \in \tr(A),$$ then $a$ is Cuntz subequivalent to $b$. 

Moreover, consider the universal minimal set $(M, \Gamma)$.
%
%or $A = \mathrm{C}(M) \rtimes \Gamma$, where $(M, \Gamma)$ is the universal minimal set of $\Gamma$
The simple C*-algebra $\mathrm{C}(M)\rtimes\Gamma$ is shown to be AH in the strong sense that there is an increasing net of unital sub-C*-algebras $A_\lambda \subseteq A$, $\lambda \in \Lambda$, such that each $A_\lambda$ is a simple (separable) $\mathcal Z$-absorbing approximately homogeneous C*-algebra with real rank zero and $A = \bigcup_{\lambda \in \Lambda} A_\lambda$. In particular, $\mathrm{C}(M)\rtimes\Gamma$ is approximately divisible.
 
\end{abstract}

\maketitle

\section{Introduction}

The uniform algebra of a discrete metric space was introduced by John Roe in \cite{Roe-I} and \cite{Roe-II} in connection with index theory. In the case the space is a discrete group $\Gamma$, the uniform Roe algebra was shown in \cite{Yu_1995} to be isomorphic to the reduced crossed product C*-algebra $l^\infty(\Gamma) \rtimes_{\mathrm r} \Gamma$, where the action of $\Gamma$ on $l^\infty(\Gamma)$ is induced by translation. 

The C*-regularity properties of the uniform Roe algebra have been found to be closely related to the asymptotic dimension of the underlying space and dynamics (see \cite{Yu_1998}, \cite{WZ-ndim}, \cite{Li_2018}, \cite{LW-Roe}, \cite{Wei-Roe}, \cite{Scarparo-Roe}, \cite{R_RDAM_2011}, \cite{ES-Roe}, \cite{GYW_2024}, etc.). In this paper, we shall focus on the group case, and so on  the crossed product  C*-algebra $l^\infty(\Gamma) \rtimes\Gamma$, where $\Gamma$ is a countable discrete  amenable group. We shall show that this C*-algebra (recall that the full and reduced crossed products are the same for an amenable group) has the strict comparison property: 
\theoremstyle{theorem}
\newtheorem{thmA}[thm]{Theorem}
\begin{thmA}[Theorem \ref{cp-gp}]
Let $\Gamma$ be a countable discrete amenable group, and let $A = l^\infty(\Gamma) \rtimes \Gamma$. If $a, b \in A \otimes \mathcal K$ are positive elements such that $$\mathrm{d}_\tau(a) < \mathrm{d}_\tau(b),\quad \tau \in \tr(A),$$ then $a$ is Cuntz subequivalent to $b$.
\end{thmA}

%Note that in the case that $\Gamma$ is not amenable, it was shown in \cite{} (Corollary 5.6) that every projection of $l^\infty(\Gamma)$ which is full in $A:=l^\infty(\Gamma) \rtimes \Gamma$ is properly infinite in the sense that $a \oplus a \precsim a$. 

To show the strict comparison stated above, we consider the topological dynamical system $(\widehat{l^\infty(\Gamma)},  \Gamma)$, where $\widehat{l^\infty(\Gamma)}$ is the spectrum of the commutative C*-algebra $l^\infty(\Gamma)$ (which is homeomorphic to $\beta\Gamma$, the Stone-\v{C}ech compactification of $\Gamma$), and consider the uniform Rohklin property, URP,  and the relative comparison property, (COS), which were studied in \cite{Niu-MD-Z}. These two properties together imply in general that the crossed-product C*-algebra can be weakly tracially approximated by the sub-C*-algebras from the Rohklin towers, which are homogeneous C*-algebras (see \cite{Niu-MD-Z}). This pair of  properties is known to hold for arbitrary free and minimal $\Int^d$-actions and for extensions of a free action of a group with subexponential growth on the Cantor set (\cite{Niu-MD-Zd}, \cite{Niu-MD-Z}). %, and then the comparison property of the crossed-product C*-algebra can be reduced to the comparison property of 

We shall observe that the URP and property (COS) also hold for $(\beta\Gamma,  \Gamma)$. Since $\Gamma$ is amenable, by \cite{DHZ-tiling}, $\Gamma$ has (exact) tilings by F{\o}lner sets. This in fact implies that the dynamical system $(\beta\Gamma, \Gamma)$ has a strong version of the URP in the sense that the Rohklin towers form a partition of unity (Proposition \ref{S-URP}). It then follows that the crossed-product C*-algebra $l^\infty(\Gamma) \rtimes \Gamma$ can be weakly tracially approximated by unital homogeneous C*-algebras (with zero-dimensional base spaces) (Proposition \ref{urp}). This strong version of the URP also implies the property (COS) for the sub-C*-algebra $l^\infty(\Gamma)$ inside $l^\infty(\Gamma) \rtimes \Gamma$. Once the URP and property (COS) are established, strict comparison  for $l^\infty(\Gamma) \rtimes \Gamma$ then follows from an argument similar to that for Theorem 4.8 of \cite{Niu-MD-Z}.

The C*-algebra $l^\infty(\Gamma) \rtimes \Gamma$ is not simple if $\Gamma$ is an infinite amenable group. In fact, by \cite{Chou_1969}, the dynamical system $(\beta\Gamma, \Gamma)$ has at least $2^c$ minimal (non-empty) closed invariant subsets. Moreover, by \cite{Ellis_1960}, all of these minimal subsets are  universal (and hence are isomorphic). Note that since the action of $\Gamma$ on $\beta\Gamma$, and hence on these minimal subsets, is free, by \cite{Elliott_1980} (Theorem 3.2 and Remark 3.7), the reduced crossed products are simple. Since $\Gamma$ is amenable, so are the full crossed products, and hence these are just the crossed products arising as the quotients of $l^\infty(\Gamma) \rtimes \Gamma$. In particular, all of the (non-zero) simple quotients of the C*-algebra $l^\infty(\Gamma) \rtimes \Gamma$ are isomorphic to $\mathrm{C}(M) \rtimes \Gamma$, where $(M, \Gamma)$ is the universal minimal set of $\Gamma$. 

Note that  the URP of $(\beta\Gamma, \Gamma)$ carries over to the subsystem $(M, \Gamma)$, and hence $\mathrm{C}(M) \rtimes \Gamma$ also has strict comparison (a different proof of this was indicated by Suzuki in \cite{Suzuki-sr1}, and also of the properties of stable rank one and real rank zero of $\mathrm{C}(M) \rtimes \Gamma$; see Corollary 1.3 and Remark 4.3 there). 
In fact, using the classification theorem, one actually can show that this simple C*-algebra is an increasing union of separable simple $\mathcal Z$-absorbing AH algebras with real rank zero. In particular, it is approximately divisible:
\theoremstyle{theorem}
\newtheorem{thmB}[thm]{Theorem}
\begin{thmB}[Theorem \ref{min-alg} and Corollary \ref{AH-structure}]
Let $\Gamma$ be a countable discrete amenable group, and denote by $(M, \Gamma)$ the universal minimal set of $\Gamma$. Then there is an increasing net $A_\lambda$, $\lambda \in \Lambda$, of sub-C*-algebras of $A: = \mathrm{C}(M) \rtimes \Gamma$ such that 
each $A_\lambda$, $\lambda \in \Lambda$, is a separable simple $\mathcal Z$-absorbing AH algebra with real rank zero, and $$\bigcup_{\lambda \in \Lambda} A_\lambda  = A. $$
In particular, $A$ is approximately divisible.
\end{thmB}

%The strict comparison of $\mathrm{C}(M)\rtimes \Gamma$ actually was point out in \cite{}, as well as the stable rank and real rank.

%\subsection*{Acknowledgements} %The research of the first named author was supported by a Natural Sciences and Engineering Research Council of Canada (NSERC) Discovery Grant, the research of the third named author was supported by a Simons Foundation grant (MP-TSM-00002606), and the research of the forth named author was supported 
%
%Part of the research was carried out when the third named author visited the Fields Institute during the spring break of 2026. He thanks the Fields Institute for the hospitality.

% It worth pointing out that, in general, unlike the case of simple C*-algebras, there are positive elements $a, b \in l^\infty(\Gamma) \rtimes\Gamma$ such that $a$ is Cuntz subequivalent to $b$, $a$ is not Cuntz equivalent to $b$, but $\mathrm d_\tau(a) = \mathrm d_\tau(b)$ for all $\tau \in \mathrm{T}(A)$. So the strict order of the Cuntz semigroup of $l^\infty(\Gamma) \rtimes \Gamma$ is still not fully determined by the traces of the C*-algebra.

%since $l^\infty(\Gamma) \rtimes \Gamma$ is not simple, 

\section{Traces and comparison}

%Let us start with some preliminaries on the traces and comparison.
\subsection{Traces}
A tracial state of a C*-algebra $A$ is a positive linear function $\tau: A \to \Comp$ such that $\norm{\tau} = 1$ and $$\tau(ab) = \tau(ba), \quad a, b \in A.$$ The set of all tracial states of $A$ is denoted by $\tr(A)$. It is a Choquet simplex if $A$ is unital.

Let $\Gamma$ be a discrete group. Then $\Gamma$ acts on itself by  (right) multiplication, and hence acts on the C*-algebra $l^\infty(\Gamma)$ (from the left). A (right) invariant mean of $\Gamma$ is a positive linear functional $\rho: l^\infty(\Gamma) \to \Comp$ such that $\norm{\rho} = 1$ and $$ \rho(\gamma \phi) = \rho(\phi), \quad \phi \in l^\infty(\Gamma),\ \gamma \in \Gamma.$$ The group $\Gamma$ is said to be amenable if it has invariant means.

Note that the restriction of a tracial state of the reduced crossed product C*-algebra $l^\infty(\Gamma)\rtimes_\mathrm{r}\Gamma$ to $l^\infty(\Gamma)$ is an invariant mean of $\Gamma$. Also note that any invariant mean of $\Gamma$ can be extended to a trace on $l^\infty(\Gamma) \rtimes_\mathrm{r} \Gamma$. Hence, the C*-algebra $l^\infty(\Gamma) \rtimes_\mathrm{r} \Gamma$ has tracial states if, and only if, $\Gamma$ is amenable.

\subsection{Cuntz subequivalence and comparison of positive elements by traces}

Let $A$ be a C*-algebra and let $a, b \in A$ be positive elements. Then $a$ is said to be Cuntz subequivalent to $b$, denote it by $a \precsim b$, if there is a sequence $(x_n) \subseteq A$ such that
$$\lim_{n \to\infty} x_n^* b x_n = a.$$

For any $\eps>0$, consider the element $(a - \eps)_+ = f(a)$, where $f(t) = \max\{t - \eps, 0\}$, $t \in \Real$. Note that $\norm{a - (a-\eps)_+} \leq \eps$. 
The following lemma will be used freely.
\begin{lem}[\cite{RorUHF2}]\label{sub-lem-Cu}
Let $A$ be a C*-algebra, and let $a, b \in A$ be positive elements. Then $a \precsim b$ if and only if $(a -\eps)_+ \precsim b$ for all $\eps>0$. 
\end{lem}

Fix $\tau \in \tr(A)$, and $a \in (A \otimes \mathcal K)^+$, and define the rank of $a$ with respect to $\tau$ as 
$$\mathrm{d}_\tau(a) = \lim_{n \to\infty} \tau(a^{\frac{1}{n}}),$$
where $\tau$ is canonically extended to a lower semicontinuous trace of $A \otimes\mathcal K$. 
For any positive elements $a, b \in A \otimes \mathcal K$, if $a \precsim b$, then $\mathrm{d}_\tau(a) \leq \mathrm{d}_\tau(b)$ for all $\tau \in \tr(A)$. % if $a \precsim b$. % $$a \precsim b \quad \Longrightarrow \quad \mathrm{d}_\tau(a) \leq \mathrm{d}_\tau(b),\quad \tau \in \tr(A).$$

The reverse implication does not hold in general, even if $\mathrm{d}_\tau(a) < \mathrm{d}_\tau(b)$ is assumed. However, for a large class of C*-algebras, such as simple exact $\mathcal Z$-absorbing C*-algebras, Cuntz subequivalence is determined by trace strict inequalities (a property referred to as strict comparison) (see \cite{Ror-Z-stable}). In this paper, we shall show that strict comparison also holds for the C*-algebra $l^\infty(\Gamma) \rtimes \Gamma$, where $\Gamma$ is a countable discrete amenable group (Corollary \ref{cp-gp}).

We shall use the following results on the comparison properties of a homogeneous C*-algebra $A=\mathrm{M}_n(\mathrm{C}(\Omega))$, where $\Omega$ is a compact Hausdorff space. In general, it  does not have strict comparison. But its diagonal subalgebra always has a certain relative comparison property, regardless of the topological dimension of $\Omega$, and it has strict comparison if the topological dimension of $\Omega$ is zero.  These statements are known for separable C*-algebras, and the general statements can be reduced as follows to the separable case. 
\begin{lem}[cf.~Theorem 7.8 of \cite{Niu-MD-Z}]\label{matrix-C}
Consider the C*-algebra $A=\mathrm{M}_n(\mathrm{C}(\Omega))$, where $n \in \mathbb{N}$ and $\Omega$ is a compact Hausdorff space. Then for any positive contractions $f = \mathrm{diag}\{f_1, ..., f_n\}$ and $g = \mathrm{diag}\{g_1, ..., g_n\}$ such that
$$\mathrm{rank}(f(x)) < \frac{1}{4} \mathrm{rank}(g(x)) \quad\mathrm{and} \quad 4 < \mathrm{rank}(g(x)), \quad x \in \Omega, $$ it follows that $f \precsim g$ in $A$. 
\end{lem}
\begin{proof}
By Theorem 7.8 of \cite{Niu-MD-Z}, the statement holds if $A$ is separable.  The general case can be reduced to the separable case: 
Pick a separable sub-C*-algebra $C \subseteq \mathrm{C}(\Omega)$ which contains $f_1,..., f_n, g_1, ..., g_n$, and consider the separable sub-C*-algebra $\mathrm{M}_n(C) \subseteq \mathrm{M}_n(\mathrm{C}(\Omega))$. Write $C = \mathrm{C}(\Omega')$, and note that on the quotient $\Omega'$ of $\Omega$, 
$$\mathrm{rank}(f(x')) < \frac{1}{4} \mathrm{rank}(g(x')), \quad x' \in \Omega'.$$ Then, by Theorem 7.8 of \cite{Niu-MD-Z}, $f \precsim g$ in $\mathrm{M}_n(C) \subseteq \mathrm{M}_n(\mathrm{C}(\Omega))$, as desired.
\end{proof}

%Note that the C*-algebra $l^\infty(\Gamma)$ has real rank zero, which implies that the topological covering dimension of $\widehat{l^\infty(\Gamma)}$ is zero.

In the case that the topological covering dimension of $\Omega$ is zero (in particular, this includes the case $\Omega = \widehat{l^\infty(\Gamma)}$, since the C*-algebra $l^\infty(\Gamma)$ has real rank zero), one actually has strict comparison:
\begin{lem}\label{matrix-C-0}
Consider the C*-algebra $A=\mathrm{M}_n(\mathrm{C}(\Omega))$, where $n \in \mathbb{N}$ and $\Omega$ is a compact Hausdorff space with zero covering dimension. Let $f, g \in A$ be positive elements such that
$$\mathrm{rank}(f(x)) < \mathrm{rank}(g(x)), \quad x\in \Omega.$$ Then $f \precsim g$ in $A$.

As a trivial special case, if $f$ and $g$ are also diagonal projections, then there is a unitary $u \in A$ such that $u(x)$ is a permutation unitary for all $x \in \Omega$, $u^*f u$ (and $ugu^*$) is a diagonal projection, and $ u^*f u \leq g. $
\end{lem}
\begin{proof}
Note that $A$ is a locally AF algebra. 
Since $A$ has real rank zero, the Cuntz classes of both $f$ and $g$ are sups of increasing sequences of classes of projections. By compactness of $\mathrm{T}(A)$, each of the projection classes of $f$, since it is strictly less on $\mathrm{T}(A)$ than the sup of the projection classes for $g$, is eventually strictly less than on $\mathrm{T}(A)$ than the individual projection classes for $g$. Since strict comparability for projections clearly holds in $A$, taking sups in $\mathrm{Cu}(A)$ we find that $f$ is Cuntz majorized by $g$.

%Taking sups, just for the projection classes of $g$, 

%For the general case, write $f = (f_{ij})$ and $g = (g_{ij})$. Since $\mathrm{C}(\Omega)$ has real rank zero, there is a separable sub-C*-algebra $C \subseteq \mathrm{C}(\Omega)$ which contains $f_{ij}$, $g_{ij}$, $i, j = 1, ..., n$, and such that $C$ also has real rank zero. Then $f, g \in \mathrm{M}_n(C)$, and the lemma follows from the separable case.

If $f$ and $g$ are diagonal projections, then, since $\Omega$ is totally disconnected and compact, there is a partition $\Omega = \Omega_1 \sqcup \Omega_2 \sqcup \cdots \sqcup \Omega_n$ such that $\Omega_i$, $i=1, ..., n$, are clopen sets, and the restrictions of $f$ and $g$ to each $\Omega_i$ are constant functions. Then, on each $\Omega_i$, since $\mathrm{rank}(f(x)) < \mathrm{rank}(g(x))$, there is a permutation unitary $u_i$ such that $u^*_i f u_i \leq g$ on $\Omega_i$, and the desired unitary $u$ can be obtained as $u_i$ on $\Omega_i$.
\end{proof}

\section{Tilings, (URP), and (COS)}

%\subsection{Rokhlin Towers}

\subsection{F{\o}lner towers and tilings}\label{folner-tiling}
Let $(\Omega, \Gamma)$ be a topological dynamical system, where $\Omega$ is a compact Hausdorff space and $\Gamma$ is a countable discrete group. A tower is a pair $(Z, \Gamma_0)$, where $Z \subseteq \Omega$ is a subset and $\Gamma_0 \subseteq \Gamma$ is finite, such that the subsets $$Z \gamma, \quad \gamma \in \Gamma_0,$$ are mutually disjoint. 

A tower is said to be open (closed) if $Z$ is open (closed). Clopen towers $(Z, \Gamma_0)$ correspond one-to-one to pairs $(p, \Gamma_0)$ where $p \in \mathrm{C}(\Omega)$ is a projection such that the projections $$\gamma(p), \quad \gamma \in \Gamma_0,$$ are mutually orthogonal.

Using (exact) tilings of amenable groups, we are able to construct a partition of unity using towers for the (C*-)dynamical system $(l^\infty(\Gamma), \Gamma)$:
\begin{defn}
A tiling of a countable discrete group $\Gamma$ consists of finite sets (shapes) $$\Gamma_1, ..., \Gamma_n \subseteq \Gamma$$ and centres $$\underbrace{c^{(1)}_1, c^{(1)}_2, ... }_{\textrm{for $\Gamma_1$}},\  \underbrace{c^{(2)}_1, c^{(2)}_2, ... }_{\textrm{for $\Gamma_2$}},\ ..., \underbrace{c^{(n)}_1, c^{(n)}_2, ... }_{\textrm{for $\Gamma_n$}} \in \Gamma, $$ such that the subsets 
$$ \Gamma_ic^{(i)}_j, \quad i=1, ..., n,\ j=1, 2, ... , $$
are mutually disjoint, and
$$\bigcup_{i=1}^n \bigcup_{j=1}^\infty\Gamma_ic^{(i)}_j = \Gamma. $$
The finite sets $\Gamma_i$, $i=1, ..., n$, are called the shapes, and the elements  $c_j^{(i)}$, $i=1, ..., n$, $j=1, 2, ...$, are called the centres (of $\Gamma_i c_j^{(i)}$).
\end{defn}

Let $K \subseteq \Gamma$ be a finite set and let  $\eps>0$.
A (non-empty) finite set $E \subseteq \Gamma$ is said to be $(K, \eps)$-invariant if
$$\frac{\abs{EK \Delta E}}{\abs{E}} < \eps.$$ Recall that the group $\Gamma$ is amenable if, and only if, there is a sequence $(\Gamma_n)$ of (non-empty) finite subsets of $\Gamma$ such that for any $(K, \eps)$, $\Gamma_n$ is $(K, \eps)$-invariant if $n$ is sufficiently large. The sequence $(\Gamma_n)$ is called a
F{\o}lner sequence, and the finite sets $\Gamma_n$, $n=1, 2, ...$, are called F{\o}lner sets.

Tilings of a discrete amenable group by F{\o}lner sets always exist:
\begin{thm}[Theorem 4.3 of \cite{DHZ-tiling}]\label{tiling-exists}
Let $\Gamma$ be a countable discrete amenable group. For any $(K, \eps)$, where $ K \subseteq \Gamma$ is a finite set and $\eps>0$, there exist $(K, \eps)$-invariant finite sets $\Gamma_1, ..., \Gamma_S \subseteq \Gamma$ which tile the group $\Gamma$.
\end{thm}

A tiling of $\Gamma$ naturally induces a partition of unity of $l^\infty(\Gamma)$ by clopen towers: Let $\{\Gamma_i, c_j^{(i)}: i=1, ..., n,\ j=1, 2, ... \}$ be a tiling of $\Gamma$. Without loss of generality, one may assume that $e \in \Gamma_i$, $i=1, ..., n$. Define $$p_i = \chi_{\{c^{(i)}_1,\   c^{(i)}_2,\ ... \} } \in l^\infty(\Gamma),\quad i=1, ..., n.$$ Then $p_i$ is a projection and 
$$(p_i, \Gamma_i)$$
is a tower. Note that
$$ \sum_{\gamma \in \Gamma_i} \gamma(p_i) = \chi_{ \bigcup_{j=1}^\infty \Gamma_i c_j^{(i)} }. $$ Since $\{\Gamma_i, c_j^{(i)}: i=1, ..., n,\ j=1, 2, ... \}$ is a tiling of $\Gamma$, one has
$$ \sum_{i=1}^n \sum_{\gamma \in \Gamma_i} \gamma(p_i) = \mathbf \chi_\Gamma = \mathbf 1_{l^\infty(\Gamma)}. $$

\subsection{Rokhlin Partitions}

Recall (\cite{Niu-MD-Z}) that a dynamical system $(X, \Gamma)$ has the Uniform Rokhlin Property (abbreviated URP) if for any finite set $K \subseteq \Gamma$ and any $\eps>0$, there are mutually disjoint open towers $$(Z_s, \Gamma_s),\quad s=1, ..., S,$$
such that
\begin{enumerate}
\item each $\Gamma_s$, $s=1, ..., S$, is $(K, \eps)$-invariant, and
\item $\mu(X \setminus \bigcup_{s=1}^S \bigcup_{\gamma \in \Gamma_s} Z_s\gamma) < \eps$ for every invariant probability Borel measure $\mu$.
\end{enumerate}

Consider the dynamical system $(l^\infty(\Gamma), \Gamma)$. Then, by Theorem \ref{tiling-exists} and the discussion at the beginning of this section (Subsection \ref{folner-tiling}), it actually has the following stronger version of the URP, namely, the Rokhlin towers form a partition of unity:
\begin{prop}\label{S-URP}
Consider the dynamical system $(l^\infty(\Gamma), \Gamma)$. For any finite set $K \subseteq \Gamma$ and any $\eps>0$, there are mutually disjoint (clopen) towers $$(p_s, \Gamma_s),\quad s=1, ..., S,$$
such that
\begin{enumerate}
\item each $\Gamma_s$ is $(K, \eps)$-invariant, and
\item $\sum_{s=1}^S \sum_{\gamma \in \Gamma_s} \gamma(p_s) = \mathbf 1_{l^\infty(\Gamma)}$.
\end{enumerate}
\end{prop}

\begin{defn}\label{defn-S-URP}
A dynamical system $(X, \Gamma)$ is said to have the strong URP if it satisfies the conclusion of the proposition above, i.e., there are Rokhlin towers which are arbitrarily close to being invariant 
and form an exact partition of unity. 
\end{defn}

\begin{rem}\label{tp-free}
Any minimal dynamical system $(X, \Gamma)$ with the URP is topologically free, and so by \cite{Elliott_1980} (Theorem 3.2 and Remark 3.7), the reduced crossed product is simple. To see this,  let $(Z_{s, n}, \Gamma_{s, n})$, $s=1, ..., S_n$, $n=1, 2, ...$, be a sequence of Rokhlin decompositions which is arbitrarily close to being invariant and such that the leftover is also arbitrarily small as $n \to \infty$. Then, for each $\gamma \in \Gamma$, the open subset
$$ \bigcup_{n=1}^\infty  \{ x g : x \in Z_{s, n},\ g \in \Gamma_{s, n} \cap \Gamma_{s, n} \gamma^{-1},\ s=1, ..., S_n \} $$ is co-null for all invariant measures. By the minimality condition, this set must be dense. Since $\Gamma$ is countable, by the Baire category theorem, the intersection $$ \bigcap_{\gamma \in \Gamma} \bigcup_{n=1}^\infty  \{ x g : x \in Z_{s, n},\ g \in \Gamma_{s, n} \cap \Gamma_{s, n} \gamma^{-1},\ s=1, ..., S_n \} $$ is dense in $X$. But every point in this last subset has trivial stabilizer. This implies that $(X, \Gamma)$ is topologically free.

%Let $(X, \Gamma)$ and $(Y, \Gamma)$ be topological dynamical systems. If $(X, \Gamma)$ is a quotient of $(Y, \Gamma)$, and if $(X, \Gamma)$ has the strong URP above, then $(Y, \Gamma)$ also has the  strong URP.
\end{rem}

By \cite{Ellis_1960}, any minimal (non-empty) closed subset $M \subseteq \beta\Gamma$ is universal in the sense that $(M, \Gamma)$ is minimal and if $(X, \Gamma)$ is a minimal dynamical system where $X$ is compact, then it is a quotient of $(M, \Gamma)$. Since the strong URP above passes to subsystems, the universal minimal system of $\Gamma$ also has the strong URP. 
\begin{cor}\label{S-URP-M}
Let $\Gamma$ be a discrete countable amenable group. Then, its universal minimal set $(M, \Gamma)$ has the strong URP.
\end{cor}
\begin{proof}
Let $M \subseteq \beta\Gamma$ be a minimal compact invariant subset. By \cite{Ellis_1960}, $(M, \Gamma)$ provides a realization of the universal minimal system of $\Gamma$. Let $K \subseteq \Gamma$ be a finite set, and let $\eps>0$. By Proposition \ref{S-URP}, there are mutually disjoint (clopen) towers $(Z_s, \Gamma_s)$, $s=1, ..., S,$
such that
\begin{enumerate}
\item each $\Gamma_s$ is $(K, \eps)$-invariant, and
\item $\bigcup_{s=1}^S \bigcup_{\gamma \in \Gamma_s} Z_s\gamma = \beta\Gamma$.
\end{enumerate}
Then $$(Z_s \cap M, \Gamma_s), \quad s=1, ..., S, $$ are the desired towers for the universal minimal system $(M, \Gamma)$.
\end{proof}

\begin{rem}
This  strong URP should be compared to the Bratteli-Vershik models studied in \cite{Put-PJM} and \cite{HPS-Cantor}. It also should be compared to the dynamical quasitiling considered in Definition 6.1 of \cite{DZ-comparison}. %in the case that $\Omega$ is totally disconnected. 
Lower semicontinuous analogues were considered in Section 8 of \cite{Niu-MD-Z}.  
\end{rem}

As in \cite{Niu-MD-Z}, this strong version of the URP implies  that the C*-algebra $l^\infty(\Gamma) \rtimes \Gamma$ can be tracially approximated by unital sub-C*-algebras from the Rokhlin towers, which are homogeneous C*-algebras (but the cutting element in general is still not a projection). The following proposition also can be compared to \cite{Li_2018} in which the C*-algebra  $l^\infty(\Gamma) \rtimes \Gamma$ is shown to be (non-separable) AF if $\Gamma$ is locally finite.

\begin{prop}\label{urp}
Let $(\Omega, \Gamma)$ be a topological dynamical system which has the strong URP of Definition \ref{defn-S-URP} (i.e., Rokhlin towers form a partition of unity). 
%Let $\Gamma$ be a discrete amenable group. 
Then the C*-algebra  $A=\mathrm{C}(\Omega) \rtimes \Gamma$ has the following property:

For any finite subset $\{f_1, ..., f_n\} \subseteq A$ and any $\eps>0$, there exist $f'_1, ..., f'_n \in A$, a unital sub-C*-algebra $C \subseteq A$ (so $1_C = 1_A$), and a positive contraction $h \in C \cap \mathrm{C}(\Omega)$ such that (with a slight abuse of notation)  
\begin{enumerate}
\item 
$C = \bigoplus_{s = 1}^S \mathrm{M}_{n_s}(\mathrm{C}(Z_s))$, where $Z_s \subseteq \Omega$, $s=1, ..., S$, are mutually disjoint clopen subsets,
\item $\mathrm{C}(\Omega) \subseteq C$, and $\mathrm{C}(\Omega)$ is equal to the diagonal sub-C*-algebra of $\bigoplus_{s = 1}^S \mathrm{M}_{n_s}(\mathrm{C}(Z_s))$,
\item $\norm{f_i - f'_i} < \eps$, $i=1, ..., n$,
\item $\norm{[h, f'_i]} < \eps$, $ i =1, ..., n$, 
\item $hf'_ih \in C$, $i=1, ..., n$, and 
\item $\mathrm{d}_\tau(1-h) < \eps$, $\tau \in \tr(A)$.
\end{enumerate}
\end{prop}

\begin{proof}

The proof is similar to the proof of Theorem 3.8 of \cite{Niu-MD-Z} (or Lemma 3.1 of \cite{Niu-MD-Z-absorbing}), but without dealing with mean dimension.

Without loss of generality, one may assume 
$$f_i=\sum_{\gamma \in\mathcal N} f_{i, \gamma}u_\gamma$$
for a finite set $\mathcal N \subseteq \Gamma$ with $e\in \mathcal N = \mathcal N^{-1}$, and $f_{i, \gamma} \in \mathrm{C}(\Omega)$, $\gamma \in \mathcal N$. (Then the elements $f'_i$, $i=1, ..., n$ in the statement of the proposition will be $f_i$, $i=1, ..., n$.)

Set
$$\max\{1, \norm{f_{i, \gamma}}: i=1, ..., n, \gamma\in\mathcal N\} = M.$$
Pick a natural number $$L > \frac{M\abs{\mathcal N}}{\eps},$$ and choose  a sufficiently large finite set $K\subseteq\Gamma$ and a sufficiently small  positive number $\delta$ that if a finite set $\Gamma_0\subseteq\Gamma$ is $(K, \delta)$-invariant, then
\begin{equation}\label{very-small-bd}
\frac{\abs{\Gamma_0 \setminus \mathrm{int}_{\mathcal N^{L+1}} (\Gamma_0)} }{\abs{\Gamma_0}} < \frac{\eps}{2},
\end{equation}
where $\mathrm{int}_F(E) := \{\gamma\in E, \gamma F \subseteq E\}$ for any finite sets $E, F \subseteq \Gamma$.

Since $(\Omega, \Gamma)$ has the stronger URP of Proposition \ref{S-URP}, there exist clopen sets $Z_1, Z_2, ..., Z_S \subseteq \Omega$ and $(K, \delta)$-invariant sets $\Gamma_1, \Gamma_2, ..., \Gamma_S \subseteq \Gamma$ such that the subsets 
$$Z_s\gamma,\quad \gamma\in \Gamma_s,\ s=1, ..., S, $$
are mutually disjoint and
\begin{equation}\label{partition-equ}
\Omega = \bigsqcup_{s=1}^S\bigsqcup_{\gamma\in \Gamma_s} Z_s\gamma.
\end{equation}

Consider the sub-C*-algebra
\begin{equation}\label{defn-C}
C:=\mathrm{C}^*\{ u^*_\gamma f: f\in\mathrm{C}(Z_s),\ \gamma\in \Gamma_s,\ s=1, 2, ..., S\}\subseteq \mathrm{C}(\Omega) \rtimes \Gamma,
\end{equation} 
which, by Lemma 3.11 of \cite{Niu-MD-Z}, is isomorphic to $$\bigoplus_{s=1}^S\mathrm{M}_{\abs{\Gamma_s}}(\mathrm{C}(Z_s)).$$
Since the base sets $Z_s$, $s=1, ..., S$, are clopen and the towers form a partition of unity (\eqref{partition-equ}), one has 
\begin{equation}\label{all-in-C}
\mathrm{C}(\Omega) \subseteq C.
\end{equation} 
Moreover, the isomorphism constructed in the proof of Lemma 3.11 of \cite{Niu-MD-Z}  sends $\mathrm{C}(\Omega)$ onto the diagonal subalgebra of $\bigoplus_{s=1}^S\mathrm{M}_{\abs{\Gamma_s}}(\mathrm{C}(Z_s))$. This shows (1) and (2).

For each $\Gamma_s$, $s=1, 2, ..., S$, define the subsets 
$$
\left\{
\begin{array}{lll}
\Gamma_{s, L+1} & = & \mathrm{int}_{\mathcal N^{L+1}} (\Gamma_s), \\
\Gamma_{s, L} & = & \mathrm{int}_{\mathcal N^{L}} (\Gamma_s) \setminus \mathrm{int}_{\mathcal N^{L+1}} (\Gamma_s), \\
\Gamma_{s, L-1} & = & \mathrm{int}_{\mathcal N^{L-1}} (\Gamma_s) \setminus \mathrm{int}_{\mathcal N^{L}} (\Gamma_s), \\
\vdots & \vdots & \vdots \\
\Gamma_{s, 0} & = & \Gamma_s \setminus \mathrm{int}_{\mathcal N} (\Gamma_s).
\end{array}
\right.
$$

Then, for any $\gamma \in\mathcal N$, one has 
\begin{equation}\label{action-nbhd}
\Gamma_{s, l}\gamma \subseteq  \Gamma_{s, l-1} \cup \Gamma_{s, l} \cup \Gamma_{s, l+1},\quad 1\leq l \leq L.  \end{equation}

Indeed,  pick an arbitrary $\gamma'\in \Gamma_{s, l}$. By construction, one has 
\begin{equation}\label{eq-contain} 
\gamma' \mathcal N^{l} \subseteq \Gamma_s\quad\mathrm{but}\quad \gamma' \mathcal N^{l+1} \nsubseteq \Gamma_s.
\end{equation} 
Therefore, $$ \gamma'\gamma \mathcal N^{l-1}  \subseteq \gamma' \mathcal N^l \subseteq \Gamma_s, $$ and hence $\gamma'\gamma \in \mathrm{int}_{\mathcal N^{l-1}}\Gamma_s$ (since $e\in\mathcal N^{l-1}$). 

Thus, to show \eqref{action-nbhd}, one only has to show that $\gamma'\gamma \notin \mathrm{int}_{\mathcal N^{l+2}}\Gamma_s$. Suppose $\gamma'\gamma\mathcal N^{l+2} \subseteq \Gamma_s$. Since $\mathcal N$ is symmetric, one has $\gamma^{-1}\in \mathcal N$; hence $\mathcal N^{l+1} \subseteq \gamma \mathcal N^{l+2}$ and 
$$\gamma' \mathcal N^{l+1} \subseteq \gamma'\gamma \mathcal N^{l+2} \subseteq \Gamma_s,$$ which contradicts \eqref{eq-contain}.

Also note that
\begin{equation}\label{action-nbhd-L-end}
\Gamma_{s, L+1}\gamma \subseteq  \Gamma_{s, L+1} \cup \Gamma_{s, L}.
\end{equation}

For each $\gamma \in \Gamma_s$, define
$$\ell(\gamma) = l,\quad\textrm{if $\gamma \in \Gamma_{s, l}$}.$$ By \eqref{action-nbhd} and \eqref{action-nbhd-L-end}, the function $\ell$ satisfies 
\begin{equation}\label{action-nbhd-fn}
\abs{\ell(\gamma'\gamma) - \ell(\gamma)} \leq 1,\quad \gamma'\in \mathcal N,\ \gamma\in \Gamma_{s, 1}\cup\cdots\cup\Gamma_{s, L+1}.
\end{equation}

Consider the function 
$$h := \sum_{s=1}^S\sum_{l=1}^{L+1} \sum_{\gamma\in \Gamma_{s, l}} \frac{l-1}{L}(\chi_{Z_s}\circ \gamma^{-1}) = \sum_{s=1}^S \sum_{l=1}^{L+1} \sum_{\gamma\in \Gamma_{s, l}} \frac{l-1}{L} u^*_\gamma\chi_{Z_s} u_\gamma \in\mathrm{C}(\Omega)\cap C.$$

By \eqref{very-small-bd}, for all $\mu \in \mathcal M_1(\Omega, \Gamma)$,
$$
 \mu(X \setminus h^{-1}(1))  \leq  \max\{\frac{\abs{\Gamma_s \setminus \mathrm{int}_{\mathcal N^{L+1}} (\Gamma_s)} }{\abs{\Gamma_s}} : s=1, ..., S \}  < \eps,
$$
and therefore
$$\mathrm{d}_{\tau}(1-h) < \eps,\quad \tau \in \mathrm{T}(A).$$
This shows $(6)$.

Note that, by the construction of $C$ (see \eqref{defn-C}), $$\chi_{Z_s} u_\gamma \in C,\quad \gamma \in\Gamma_s.$$
Hence, for each $\gamma'\in \mathcal N$, since $\gamma\gamma' \in \Gamma_s$, $\gamma\in \Gamma_{s, l}$, $l=1, 2, ..., L+1$, one has
$$ hu_{\gamma'} = \sum_{s=1}^S \sum_{l=1}^{L+1} \sum_{\gamma\in \Gamma_{s, l}} \frac{l-1}{L} u^*_\gamma\chi_{Z_s} u_{\gamma\gamma'}  = \sum_{s=1}^S \sum_{l=1}^{L+1} \sum_{\gamma\in \Gamma_{s, l}} \frac{l-1}{L} (u^*_\gamma\chi_{Z_s}) (\chi_{Z_s} u_{\gamma\gamma'})  \in C,$$
and therefore, 
$$hu_\gamma h \in C,\quad \gamma \in \mathcal N.$$ 
Hence, by \eqref{all-in-C},  
$$hf_i h\in C,\quad 1 \leq i\leq n.$$
This shows (5).

Note that, for each $\gamma' \in\mathcal N$, by \eqref{action-nbhd-fn}, 
\begin{eqnarray*}
&& \norm{u^*_{\gamma'}hu_{\gamma'} - h}  \\
& = & \norm{ \sum_{s=1}^S\sum_{l=1}^{L+1} \sum_{\gamma\in \Gamma_{s, l}} \frac{l-1}{L}\chi_{Z_s}\circ (\gamma'\gamma)^{-1}  -  \sum_{s=1}^S\sum_{l=1}^{L+1} \sum_{\gamma\in \Gamma_{s, l}} \frac{l-1}{L} \chi_{Z_s}\circ \gamma^{-1}} \\
& = & \max\{ \abs{ \frac{\ell(\gamma'\gamma)-1}{L} - \frac{\ell(\gamma)-1}{L}}: \gamma\in\Gamma_s\setminus\Gamma_{s, 0},\ s=1, 2, ..., S \}\\
& < & \frac{1}{L}  <  \frac{\eps}{M\abs{\mathcal{N}}},
\end{eqnarray*}
and hence
\begin{equation}\label{pre-comm-p}
\norm{hf_i - f_i h} < {\eps},\quad i=1, 2, ..., n.
\end{equation}
This shows (4).
\end{proof}

\subsection{Property (COS)}
Recall (\cite{Niu-MD-Z}) that a topological  dynamical system $(X, \Gamma)$ is said to have $(\lambda, m)$-Cuntz comparison on open sets, where $\lambda \in (0, +\infty)$ and $m \in \mathbb N$, if for all $f, g \in \mathrm{C}(X)^+$ such that
$$\mathrm{d}_\tau(f) < \lambda \mathrm{d}_\tau(g), \quad \tau \in \mathcal M_1(X, \Gamma),$$
one has $f \precsim 1_m \otimes g $ in $\mathrm{C}(X) \rtimes \Gamma$, %, where $\varphi_E$ (or $\varphi_F$) is a positive continuous function on $X$ with open support $E$ (or $F$), and 
where $\mathcal M_1(X, \Gamma)$ is the set of all invariant probability Borel measures on $X$. (The property (COS) in \cite{Niu-MD-Z} was formulated using open sets.)

%Recall (\cite{Niu-MD-Z}) that a topological  dynamical system $(X, \Gamma)$ is said to have $(\lambda, m)$-Cuntz comparison on open sets, where $\lambda \in (0, +\infty)$ and $m \in \mathbb N$, if for all open sets $E, F \subseteq X$ such that
%$$\mu(E) < \lambda \mu(F), \quad \mu \in \mathcal M_1(X, \Gamma),$$
%one has $\varphi_E \precsim 1_m \otimes \varphi_F $ in $\mathrm{C}(X) \rtimes \Gamma$, where $\varphi_E$ (or $\varphi_F$) is a positive continuous function on $X$ with open support $E$ (or $F$), and $\mathcal M_1(X, \Gamma)$ is the set of all invariant probability Borel measures on $X$.

The property (COS) can be considered for general pairs of C*-algebras: Let $A$ be a unital C*-algebra and let $D \subseteq A$ be a unital sub-C*-algebra. Then $(D, A)$ is said to have the property $(\lambda, m)$-(COS) if for all positive contractions $a, b \in D$ such that
$$\mathrm{d}_\tau(a) < \lambda \mathrm{d}_\tau(b), \quad \tau \in \tr(A), $$
one has $a \precsim 1_m \otimes b $ in $A$. 

One should compare the property (COS) to the dynamical comparison property considered in \cite{KS-comparison} and \cite{DZ-comparison}. It is straightforward to check that the dynamical comparison property implies the property (COS), but it is not clear if the converse holds. 

%Since the approximating sub-C*-algebras of Proposition \ref{urp} are unital, using Lemma \ref{matrix-C}, one obtains the following (COS) property:

With the Rokhlin partition property of Proposition \ref{S-URP}, the property (COS)  follows directly from Corollary 8.8 of \cite{Niu-MD-Z} or Theorem 6.3 of \cite{DZ-comparison} for the case that $\Omega$ is totally disconnected. For the reader's convenience, we provide a proof below.

\begin{prop}[Corollary 8.8 of \cite{Niu-MD-Z}, Theorem 6.3 of \cite{DZ-comparison}]\label{cos}
Let $(\Omega, \Gamma)$ be a dynamical system which has the strong URP of Definition \ref{defn-S-URP} (i.e., Rokhlin towers form a partition of unity). Then the pair of C*-algebras $(\mathrm{C}(\Omega), \mathrm{C}(\Omega) \rtimes \Gamma)$  has the property $(\frac{1}{4}, 1)$-(COS).

Moreover, if $\Omega$ is totally disconnected (as in the case that $\Omega = \widehat{l^\infty(\Gamma)} = \beta\Gamma$ or $\Omega$ is the universal minimal set of $\Gamma$), $(\mathrm{C}(\Omega), \mathrm{C}(\Omega) \rtimes \Gamma)$  has the property $(1, 1)$-(COS) (indeed, it has the dynamical comparison property).

%More generally, for any $n \in \mathbb N$, and for any positive diagonal elements $a, b \in \mathrm{M}_n(\mathrm{C}(\Omega))$, 
\end{prop}
\begin{proof}
Let $f, g \in \mathrm{C}(\Omega)^+$ be positive elements such that $$\mathrm{d}_\mu(f) < \frac{1}{4} \mathrm{d}_\mu(g), \quad \mu \in \mathcal M_1(\Omega, \Gamma).$$
Denote by $E$ and $F$ the open supports of $f$ and $g$ respectively. Then 
%Let $E, F \subseteq \Omega$ be open sets such that
\begin{equation}\label{mean-condition}
\mu(E) < \frac{1}{4} \mu(F), \quad  \mu \in \mathcal M_1(\Omega, \Gamma). 
\end{equation}
To prove the assertion, it is enough to show that for any $\eps>0$, one has
$$(f - \eps)_+ \precsim g.$$ %(Recall that $\varphi_E$ (or $\varphi_F$) is a positive continuous function on $X$ with $E$ (or $F$) being its open support.)

First note that $\mu(F) > 0$ for all $\mu$. Since the set of invariant measures is compact, there is $\delta>0$ such that 
\begin{equation}\label{lbd-dense-pre}
\mu(F)>\delta,\quad  \mu \in \mathcal M_1(\Omega, \Gamma). 
\end{equation}

For the given $\eps$, pick a compact set $E' \subseteq E$ such that 
$$(f - \eps)_+(x) = 0 ,\quad x \notin E'. $$ Then, we assert that there is $(K, \delta)$ such that if a finite set $\Gamma' \subseteq \Gamma$ is $(K, \delta)$-invariant, then, for all $x \in \Omega$, 
\begin{equation}\label{local-density}
\frac{1}{\abs{\Gamma'}}\abs{\{\gamma \in \Gamma': x \gamma \in E'\}} < \frac{1}{4}  \frac{1}{\abs{\Gamma'}}\abs{\{\gamma \in \Gamma': x \gamma \in F\}},
\end{equation}
and
\begin{equation}\label{lbd-dense}
\delta <  \frac{1}{\abs{\Gamma'}}\abs{\{\gamma \in \Gamma': x\gamma \in F\}}.
\end{equation}

Let us only show \eqref{local-density} (for suitable $(K, \delta)$), using \eqref{mean-condition};  \eqref{lbd-dense} can be proved in a similar way (by using \eqref{lbd-dense-pre}). 
If \eqref{local-density} did not hold (for any $(K, \delta)$), there would exist a F{\o}lner sequence $\Gamma_n$, $n=1, 2, ...$, such that for each $\Gamma_n$, there is $x_n \in \Omega$ such that 
\begin{equation}\label{eq-02-003}
\frac{1}{\abs{\Gamma_n}}\abs{\{\gamma \in \Gamma_n: x_n\gamma \in E'\}} \geq \frac{1}{4}  \frac{1}{\abs{\Gamma_n}}\abs{\{\gamma \in \Gamma_n: x_n\gamma \in F\}}.
\end{equation}

Consider the discrete measures
$$
\mu_n:=\frac{1}{\abs{\Gamma_{n}}} \sum_{\gamma \in \Gamma_{n}} \delta_{x_n \gamma},\quad n=1, 2, ...,
$$
and pick an accumulation point, say $\mu_\infty$. Since $(\Gamma_n)$ is a F{\o}lner sequence, $\mu_\infty$ is invariant under the action of $\Gamma$.  
By \eqref{eq-02-003},
$$\mu_n(F) \leq 4 \mu_n(E'),\quad n=1, 2, ...,$$
and hence (note that $F$ is open and $E'$ is closed), 
\begin{eqnarray*}
\mu_\infty(F) & \leq & \liminf_{n \to \infty} \mu_n(F) \\
& \leq & 4 \liminf_{n \to \infty} \mu_n(E') \\
& \leq & 4 \limsup_{n \to \infty} \mu_n(E') \\
& \leq & 4 \mu_\infty(E'),
\end{eqnarray*}
which contradicts \eqref{mean-condition}, and so \eqref{local-density}  holds. 

Now, with $(K, \delta)$ as assured above, by (1) and (2) of Proposition \ref{urp}, there is a unital sub-C*-algebra $\mathrm{C}(\Omega) \subseteq C \subseteq \mathrm{C}(\Omega) \rtimes \Gamma$ such that
$$C = \bigoplus_{s = 1}^S \mathrm{M}_{\abs{\Gamma_s}}(\mathrm{C}(Z_s)),$$
with $\mathrm{C}(\Omega)$ the diagonal  subalgebra, for $(K, \delta)$-invariant subsets $\Gamma_s$, $s=1, ..., S$, with $\abs{\Gamma_s} > 4/\delta$. 
By \eqref{local-density}, one has, for all $x \in \bigsqcup_{s = 1}^S Z_s$, 
\begin{eqnarray*}
\mathrm{rank}(f - \eps)_+(x) & = &\abs{\{\gamma \in \Gamma_s: (f - \eps)(x\gamma) > 0 } \\
& \leq &\abs{\{\gamma \in \Gamma_s: x\gamma \in E'  \} } \\
& < & \frac{1}{4}\abs{\{\gamma \in \Gamma_s: x\gamma \in F  \} } \\
& = & \frac{1}{4} \mathrm{rank}(g(x)).
\end{eqnarray*}  
By \eqref{lbd-dense}, 
$$ 4 <  \delta \abs{\Gamma_s} <  \abs{\{\gamma \in \Gamma_s: x\gamma \in F  \} } = \mathrm{rank}(g(x)). $$
Then, by Lemma \ref{matrix-C}, 
$$ (f - \eps)_+ \precsim g,$$
as desired.

If $\Omega$ is totally disconnected, the base sets $Z_1, ..., Z_S$ are also totally disconnected. Then, on using Lemma \ref{matrix-C-0} instead of Lemma \ref{matrix-C}, the argument above shows that the dynamical system $(\mathrm{C}(\Omega), \mathrm{C}(\Omega) \rtimes \Gamma)$ actually has the property $(1, 1)$-(COS).

If, moreover, $f$ and $g$ are projections (hence $E$ and $F$ are clopen sets), the same argument above and the second part of Lemma \ref{matrix-C-0} provide a unitary $u \in C$ such that $u$ is locally a permutation unitary, $u^*fu$ is a diagonal projection, and $$u^*fu \leq g.$$ 

Since permutations are induced by group actions and diagonal projections correspond to clopen subsets of $\Omega$, the subequivalence above implies that there are a partition $E=E_1\sqcup \cdots \sqcup E_n$ and group elements $\gamma_1, ..., \gamma_n$ such that 
$$E_i\gamma_i \subseteq F \quad \mathrm{and} \quad E_i\gamma_i \cap E_j\gamma_j = \O, \quad i, j=1, ..., n,\ i\neq j. $$ So, $(\Omega, \Gamma)$ has dynamical comparison in this case.
\end{proof}

\section{Strict comparison}
Consider a dynamical system $(\Omega, \Gamma)$ with the URP and (COS). Assume that $\Omega$ is totally disconnected (in particular, $(\Omega, \Gamma)$ has zero mean dimension). If $(\Omega, \Gamma)$ is free and minimal, and $\Omega$ is metrizable, then, by Theorem 4.8 of  \cite{Niu-MD-Z}, the C*-algebra $\mathrm{C}(\Omega) \rtimes \Gamma$ has strict comparison and is $\mathcal Z$-absorbing.

%Hence $\mathrm{C}(M)\rtimes\Gamma$ has strict comparison, where $(M, \Gamma)$ is the universal minimal set of a discrete amenable group $\Gamma$.

In this section, we shall show that strict comparison holds in a more general setting. In particular, the C*-algebras $l^\infty(\Gamma) \rtimes \Gamma$ and $\mathrm{C}(M) \rtimes \Gamma$ have strict comparison. Moreover, the simple C*-algebra $\mathrm{C}(M) \rtimes \Gamma$ is an increasing union of separable $\mathcal Z$-absorbing AH algebras with real rank zero. So, it has stable rank one and real rank zero, and is approximately divisible. (As mentioned, see also \cite{Suzuki-sr1} for the results for $M$, except approximate divisibility.)

\subsection{Strict comparison}
\begin{lem}\label{sub-dom}
Let $(\Omega, \Gamma)$ be a dynamical system which has the strong URP of Definition \ref{defn-S-URP} (i.e., Rokhlin towers form a partition of unity). Assume that $\Omega$ is totally disconnected, and $\abs{\Gamma} = \infty$.

Let $a \in A = \mathrm{C}(\Omega)\rtimes\Gamma$ be a positive contraction  such that $$\tau(a) > 0,\quad \tau \in \tr(A). $$ Then there is a positive element $h \in \mathrm{C}(\Omega)$ such that $$ h \precsim a \quad \textrm{and} \quad \tau(h) > 0, \quad \tau \in \tr(A). $$

Moreover, $h$ can be chosen such that $0$ is an isolated point of $\mathrm{sp}(h)$, and for any given $\eps>0$, $h$ can be chosen such that also $$\mathrm{d}_\tau(h) < \eps,\quad \tau \in \tr(A). $$
\end{lem}

\begin{proof}
%If $\abs{\Gamma} < +\infty$, then, by letting the shape of the Rokhlin tower to be $\Gamma$, the dynamical system $(\Omega, \Gamma)$ is isomorphic to $(Z \times \Gamma, \Gamma)$, where $Z$ is a compact Hausdorff space, and $\Gamma$ actions on $Z \times \Gamma$ by translation. Then $A \cong \mathrm{M}_{\abs{\Gamma}}(\mathrm{C}(Z))$.
%
Let $\eps>0$ be given. Since $\tr(A)$ is compact, there is $\delta>0$ such that $$\tau(\mathbb E(a)) = \tau(a) > \delta,\quad \tau \in \tr(A). $$ 
Choose $\delta' < \frac{\delta}{4}$, and choose a positive contraction $a' = \sum_{\gamma \in \mathcal G} f_\gamma u_\gamma$, where $\mathcal G \subseteq \Gamma$ is a finite set such that $\mathcal G = \mathcal G^{-1}$, such that 
\begin{equation}\label{alg-approx}
\norm{a - a'} < \delta' . 
\end{equation}
Then 
$$ \tau((\mathbb E(a') - \delta')_+) \approx_{\delta'} \tau(\mathbb E(a')) \approx_{\delta'} \tau(\mathbb E(a)) > \delta,$$
and hence
$$ \tau( (\mathbb E(a') - \delta')_+ ) > \delta - 2\delta' > \frac{\delta}{2}, \quad \tau \in \tr(A).$$

Since $(\Omega, \Gamma)$ has the strong version of the Rokhlin property,  a compactness argument (applied to $\tr(A)$) shows that there is a Rokhlin partition $(Z_s, \Gamma_s)$, $s=1, ..., S$, such that each $\Gamma_s$ is sufficiently invariant that 
\begin{equation}\label{avg-1}
\frac{1}{\abs{\Gamma_s}} \sum_{\gamma\in \Gamma_s} (\mathbb E(a') - \delta')_+(x \gamma) > \frac{\delta}{2}, \quad x\in Z_s,
\end{equation}
\begin{equation}\label{avg-2}
\frac{\abs{\partial_{\mathcal G} \Gamma_s}}{\abs{\Gamma_s}} < \frac{\delta}{4},\quad s=1, ..., S,
\end{equation}
and
\begin{equation}\label{large-G}
\abs{\Gamma_s} > \frac{1}{\eps},\quad s=1, ..., S.
\end{equation}

Since each $Z_s$ is totally disconnected and the function $(\mathbb E(a') - \delta')_+$ is continuous, by \eqref{avg-1} and \eqref{avg-2}, for each $x \in Z_s$, there are $\gamma_x \in \mathrm{int}_{\mathcal G}(\Gamma_s)$ and a clopen neighbourhood $Z_{s, x} \subseteq Z_s$ such that 
$$(\mathbb E(a') - \delta')_+(y\gamma_{x}) > \frac{\delta}{2} - \frac{\delta}{4} =\frac{\delta}{4},\quad y \in Z_{s, x}.$$
Then another compactness argument (applied to $Z_s$) shows that there exist a partition $Z_s = Z_{s, 1} \sqcup \cdots \sqcup Z_{s, N_s}$ and $\gamma_{s, 1}, ..., \gamma_{s, N_s} \in \mathrm{int}_{\mathcal G}(\Gamma_s)$ such that 
$$(\mathbb E(a') - \delta')_+(x\gamma_{s, n}) > \frac{\delta}{4},\quad x \in Z_{s, n},\ n=1, ..., N_s.$$
Therefore, the tower $(Z_s, \Gamma_s)$ can be split into thinner towers, 
$$(Z_{s, 1}, \Gamma_s), ..., (Z_{s, N_s}, \Gamma_s), $$ 
such that for each $(Z_{s, n}, \Gamma_s)$, there is $\gamma_{s, n} \in \mathrm{int}_{\mathcal G}(\Gamma_s)$ such that 
\begin{equation}\label{compact-cond}
(\mathbb E(a') - \delta')_+(x) > \frac{\delta}{4}, \quad x\in Z_{s, n}\gamma_{s, n}. \end{equation} 

Consider the projection 
$$ p := \sum_{s=1}^S\sum_{n=1}^N \chi_{Z_{s, n}\gamma_{s, n}}$$ 
corresponding to the level sets $$Z_{s, n} \gamma_{s, n}, \quad n=1, ..., N_s, \ s = 1, ..., S.$$
It follows from \eqref{large-G} that
\begin{equation}\label{small-ubd}
\tau(p) < \max\{\frac{1}{\abs{\Gamma_s}}: s=1, ..., S \} <\eps,
\end{equation}
and it follows from \eqref{compact-cond} that 
\begin{equation}\label{trace-positive} 
\tau(p(\mathbb E(a') - \delta')_+p) > \frac{\delta}{4\abs{\Gamma_s}} > 0 ,\quad \tau \in \tr(A).
\end{equation}
Note that, since $\gamma_{s, n} \in \mathrm{int}_{\mathcal G}(\Gamma_s)$ one has
$$pu_\gamma p = 0,\quad \gamma \in \mathcal G\setminus\{e\},$$
and hence 
$$pa'p = p\mathbb E(a')p.$$
Then, since $p$ is a projection and $p$ commutes with $\mathbb E(a')$, 
\begin{equation}\label{move-p-inside}
 (pa'p - \delta')_+  = ( p\mathbb E(a') p - \delta')_+ = p( \mathbb E(a') - \delta')_+p \in l^\infty(\Gamma). 
\end{equation} 

By \eqref{alg-approx}, 
$$pa'p \approx_{\delta'} pap,$$
and hence 
$$(pa'p - {\delta'})_+ \precsim pap \precsim a.$$
By \eqref{move-p-inside} and \eqref{trace-positive}, 
$$ \tau((pa'p - \delta')_+)  =  \tau( p(\mathbb E(a') - \delta')_+p) >0, \quad \tau \in \tr(A).$$
By \eqref{compact-cond}, $0$ is an isolated point of the spectrum of $(pa'p - {\delta'})_+$. By \eqref{small-ubd}, $$\tau((pa'p - \delta')_+) \leq \tau(p) < \eps,\quad \tau \in \tr(A).$$
Thus, $h:= (pa'p - {\delta'})_+$ is the desired element.
\end{proof}

\begin{lem}\label{full-element}
Let $(\Omega, \Gamma)$ be a dynamical system which has the strong URP of Definition \ref{defn-S-URP} (i.e., Rokhlin towers form a partition of unity). Assume that $\Omega$ is totally disconnected, and $\abs{\Gamma} = \infty$.

Let $A = \mathrm{C}(\Omega) \rtimes \Gamma$, and let $a \in A$ be a positive element such that $$\tau(a) > 0,\quad \tau \in \tr(A).$$ Then $a$ is full.
\end{lem}
\begin{proof}
First, observe that the projection of the base sets $p:=p_1 + \cdots +p_S$ in Proposition \ref{S-URP} is a full element of $\mathrm{C}(\Omega) \rtimes \Gamma$. Since $\abs{\Gamma} = \infty$, the trace of $p$ can be chosen to be arbitrarily small.

Let $I \subseteq A$ be a two-sided closed ideal which contains $a$. By Lemma \ref{sub-dom}, there is a projection $h \in \mathrm{C}(\Omega)$ such that 
$$ h \precsim a \quad \mathrm{and} \quad \tau(h) > 0, \quad \tau \in \tr(A). $$
In particular, $h \in I$. 

Since $\tr(A)$ is compact, there is $\delta >0$ such that $ \tau(h) > \delta$, $\tau \in \tr(A)$. By the observation made at the beginning of the proof, there is a full  projection $p \in \mathrm{C}(\Omega)$ such that $$ \tau(p) < \delta, \quad \tau \in \tr(A).$$
By the property (COS) (Proposition \ref{cos}), one has $$p \precsim  h,$$ which implies $p \in I$. Since $p$ is full, one has $I=A$.
\end{proof}

%\begin{rem}
%The dynamical system $(l^\infty(\Gamma), \Gamma)$ indeed satisfies the Kishimoto's condition, and hence the conditional expectation $l^\infty(\Gamma) \rtimes \Gamma \to l^\infty(\Gamma)$ preserves ideals, i.e., $\mathbb E(I) \subseteq I$ for all closed two-sided ideals of $l^\infty(\Gamma) \rtimes \Gamma$.
%\end{rem}

The following lemma certainly is well known:
\begin{lem}\label{cut-out-proj}
Let $A$ be a C*-algebra. Let $a \in A$ be a positive element, and let $p\in \overline{aAa}$ be a projection. Then, there are positive elements $a', q \in \overline{aAa}$ such that $q$ is a projection, $p \sim q$, $a' \perp q$, and $a$ is Cuntz equivalent to $a'+q$.
\end{lem}
\begin{proof}
Passing to a separable sub-C*-algebra of $A$ which contains $a$ and $p$, one may assume that $A$ is separable. Choose an approximate unit $(e_n)_{n=1}^\infty$ for $\overline{aAa}$ such that $e_n e_{n+1} = e_{n+1} e_n = e_n$, $n=1, 2, ...$. Since $p$ is a projection (so it is a compact element), there is a projection $q \in \overline{e_{n-1}Ae_{n-1}}$ for a sufficiently large $n$ such that $p \sim q$. Note that $e_n q = q$. Then $$a ' : = (e_{n+1} - q) + \sum_{i=1}^\infty\frac{1}{2^i}(e_{n+1+i} - e_{n+i})$$
and $q$ has the desired property.
\end{proof}

\begin{prop}\label{main-thm}
Let $(\Omega, \Gamma)$ be a dynamical system which has the strong URP of Definition \ref{defn-S-URP} (i.e., Rokhlin towers form a partition of unity), and assume that $\Omega$ is totally disconnected. 

If $a, b  \in A:=\mathrm{C}(\Omega) \rtimes \Gamma$ are positive elements such that 
\begin{equation}\label{comp-equ}
\mathrm{d}_\tau(a) < \mathrm{d}_\tau(b),\quad \tau \in \tr(A),
\end{equation} 
then $a \precsim b$.

Moreover, if $(\Omega, \Gamma)$ is minimal, then the statement above holds for all positive elements $a, b \in A \otimes \mathcal K$. 
\end{prop}

\begin{proof}
The proof is similar to the proof of Theorem 4.8 of \cite{Niu-MD-Z}.

By \cite{RC-5-person}, we may assume that $a, b \in \mathrm{M}_\infty(A)$. Let $a, b \in \mathrm{M}_n(A)$, where $n$ is an arbitrary natural number if $(\Omega, \Gamma)$ is minimal, and $n=1$ for the general case.

By Lemma \ref{sub-lem-Cu} and the compactness of $\tr(A)$, upon replacing $a$ by $(a-\eps)_+$ for an arbitrary $\eps>0$, we may assume that there is $\delta>0$ such that 
$$ \mathrm{d}_\tau(a) + \delta < \mathrm{d}_\tau(b),\quad \tau \in \tr(A). $$

Note that $\mathrm{d}_\tau(b) >0$ for all $\tau \in \tr(A)$, and hence $\tau(b) > 0$ for all $\tau \in \tr(A)$. 

In the case that $(\Omega, \Gamma)$ is minimal (so $n$ is arbitrary), since $b \in \mathrm{M}_n(A)$ is not zero, pick a non-zero diagonal element $b_0$. Since $A$ is simple in this case, one has $\tau(b_0) > 0$ for all $\tau \in \tr(A)$. Note that $b_0\precsim b$ since $b_0 = ebe$ for a rank one projection $e$. In the general case (when $n=1$), just set $b_0 = b$.

By Lemma \ref{sub-dom}, there is a projection $p' \in \mathrm{C}(\Omega)$ such that $p' \precsim b_0 \precsim b$ and 
\begin{equation}\label{eq-01001}
 0 < \tau(p') < \frac{\delta}{2}, \quad \tau \in \tr(A).
 \end{equation}
Then there is a projection $p \in \overline{b\mathrm{M}_n(A)b}$ such that $p$ is Murray-von Neumann equivalent to $p'$. %and hence
%$$ \delta_1 < \tau(p) < \frac{\delta}{2}, \quad \tau \in \tr(A).$$
%for some $\delta_1>0$. 

Pick mutually orthogonal and mutually equivalent non-zero projections $p'_1, ..., p'_n \leq p'$ (in the case that $n=1$, just choose $p_1' = p'$). Note that
\begin{equation}\label{eq-0100101}
\tau(p'_1) > \delta_1, \quad \tau \in \tr(A),
\end{equation}
for some $\delta_1>0$.

By Lemma \ref{cut-out-proj}, there are positive elements $\tilde{b}, q \in \overline{b\mathrm{M}_n(A)b}$ such that $q$ is a projection, $q$ is Murray-von Neumann equivalent to $p$, $\tilde{b} \perp q$, and $\tilde{b} + q \sim b$. Therefore, by \eqref{eq-01001}, 
\begin{equation}\label{trace-comp}
\mathrm{d}_\tau(a) + \frac{\delta}{2} < \mathrm{d}_\tau(\tilde{b}), \quad \tau \in \tr(A).
\end{equation} 

In particular, $\tau(\tilde{b}) >0$, $\tau \in \tr(A)$. By Lemma \ref{full-element}, the positive element $\tilde{b}$ is full in $A$ (note that $n=1$ for the possibly non-simple case). Thus, the Cuntz class of $\tilde{b}$ is a strong order unit of the Cuntz semigroup of $\mathrm{C}(\Omega) \rtimes \Gamma$ (i.e., for each positive element $s \in \mathrm{M}_\infty(\mathrm{C}(\Omega) \rtimes \Gamma)$, there is $m \in \mathbb N$ such that $s \precsim \tilde{b}\otimes 1_m$). By \eqref{trace-comp} and the proof of Proposition 3.2 of \cite{RorUHF2}, there is $N \in \mathbb N$ such that
\begin{equation}\label{mult-comp}
  a \otimes 1_{N+1} \precsim  \tilde{b} \otimes 1_N. 
\end{equation}

Let $\eps>0$ be arbitrary. Then, there is $(x_{i, j})_{i, j=1,..., N+1}$ such that
\begin{equation}\label{close-relation}
\norm{a \otimes 1_{N+1} - (x_{i, j})^*(\tilde{b} \otimes 1_N)(x_{i, j})} < \eps.
\end{equation}

Denote by $\delta' := \delta(\eps, \norm{(x_{i, j})}, N+1)$ the constant of Lemma 4.6 of \cite{Niu-MD-Z} with respect to $\norm{(x_{i, j})}$ and $\eps$.

Since $(\mathrm{C}(\Omega), \Gamma)$ has the URP, by Proposition \ref{urp}, for any $\eps'>0$ (to be determined), there exist a unital sub-C*-algebra $C$, $h \in C$, $a'$, $b'$, and $x_{i, j}'$, $i, j =1, ..., N+1$, in $\mathrm{M}_n(A)$ such that
$$C \cong \bigoplus_{s=1}^S \mathrm{M}_{n_s}(\mathrm{C}(Z_s)),$$ where $Z_s \subseteq \Omega$, $s=1, ..., S$, are clopen subsets, 
$$ \norm{x'_{i, j} - x_{i, j}},\  \norm{a' - a},\  \norm{b' - \tilde{b}} < \eps' < \frac{\eps^3}{16 \max\{1, \norm{(x_{i, j})}^6\}}, \quad i, j = 1, ..., N+1,$$
\begin{equation}\label{almost-comm-0} 
\norm{[x'_{i, j}, h\otimes 1_n]},\  \norm{[a', h\otimes 1_n]},\  \norm{[b', h\otimes 1_n]} < \eps' < \delta', \quad i, j = 1, ..., N+1,
\end{equation}
\begin{equation}\label{in-C-01} 
hx_{i, j}'h, ha'h, hb'h \in \mathrm{M}_n(C),\quad i, j=1, ..., N+1, 
\end{equation}
and
$$\mathrm{d}_\tau(1 - h) < \delta_1,\quad \tau \in \tr(A).$$

Then $\eps'$ can be chosen to be sufficiently small that
\begin{equation}\label{appro-equ-1}
\norm{a' \otimes 1_{N+1}  - (x'_{i, j})^* (b' \otimes 1_N) (x'_{i, j})} < \eps,
\end{equation}
and
\begin{equation}\label{appro-equ-2}
\norm{b' - \tilde{b}} < \frac{\eps^3}{8 \max\{1, \norm{(x'_{i, j})}^6\}}.
\end{equation}

By \eqref{eq-0100101}, $$\mathrm d_\tau(1-h) < \delta_1 < \mathrm{d}_\tau(p'_1),\quad \tau \in \tr(A),$$
and by the property (COS) (Proposition \ref{cos}), 
\begin{equation}\label{eq-01002}
(1 - h)\otimes 1_n \precsim p' \sim p \sim q.
\end{equation}

By \eqref{appro-equ-1}, \eqref{almost-comm-0} and \eqref{in-C-01}, it follows from Lemma 4.6 of \cite{Niu-MD-Z} that there is $c \in \mathrm{M}_n(C)$ such that
$$ ((h\otimes 1_n)a'(h\otimes 1_n) - 102\eps)_+\otimes 1_{N+1} \precsim_C c \otimes 1_N$$
and
\begin{equation}\label{eqn-26}
c \precsim (h \otimes 1_n )(b' - \frac{\eps^3}{8 \max\{1, \norm{x'}^6\}})_+(h\otimes 1_n).
\end{equation}
In particular
$$ \mathrm{rank}(((h\otimes 1_n)a'(h\otimes 1_n) - 102\eps)_+(x)) \leq \mathrm{rank}(c(x)), \quad x\in \widehat{\mathrm{M}_n(C)} = \bigsqcup_{s = 1}^S Z_s.$$
Since $\Omega$ is zero dimensional, the clopen sets $Z_1, ..., Z_S$ are also zero dimensional. By Lemma \ref{matrix-C-0}, 
%$\widehat{C}$ is zero dimensional (so it has comparison), 
one has 
\begin{equation}\label{eqn-27}
((h\otimes 1_n)a'(h\otimes 1_n) - 102\eps)_+ \precsim_C c.
\end{equation}

Then
\begin{eqnarray*}
a 
& \approx_{\eps} &
a' \\
& \approx_\eps &
((1-h)^{\frac{1}{2}} \otimes 1_n)a'((1-h)^{\frac{1}{2}} \otimes 1_n) + (h^{\frac{1}{2}} \otimes 1_n)a'(h^{\frac{1}{2}} \otimes 1_n) \\
& \approx_{102\eps} &
((1-h)^{\frac{1}{2}} \otimes 1_n)a'((1-h)^{\frac{1}{2}} \otimes 1_n) + ((h^{\frac{1}{2}} \otimes 1_n)a'(h^{\frac{1}{2}} \otimes 1_n) - 102\eps)_+,
\end{eqnarray*}
and by \eqref{eqn-27}, \eqref{eq-01002}, \eqref{eqn-26},  and \eqref{appro-equ-2},
\begin{eqnarray*}
&& ((1-h)^{\frac{1}{2}} \otimes 1_n)a'((1-h)^{\frac{1}{2}} \otimes 1_n) + ((h^{\frac{1}{2}} \otimes 1_n)a'(h^{\frac{1}{2}} \otimes 1_n) - 102\eps)_+\\
  & \precsim & (1-h) \otimes 1_n + c \\
& \precsim &q \oplus (b' - \frac{\eps^3}{8 \max\{1, \norm{x'}^6\}})_+ \\
& \precsim & q + \tilde{b} \precsim b,
\end{eqnarray*}
which implies $(a - 104\eps)_+ \precsim b$. Since $\eps$ is arbitrary, one has $a \precsim b$, as desired.
\end{proof}

Since $l^\infty(\Gamma)$ has real rank zero, its spectrum $\widehat{l^\infty(\Gamma)} = \beta\Gamma$ is totally disconnected. Therefore, we have the following corollary:
\begin{thm}\label{cp-gp}
Let $\Gamma$ be a countable discrete amenable group. Let $A = l^\infty(\Gamma) \rtimes \Gamma$ or $A = \mathrm{C}(M)\rtimes\Gamma$, where $M$ is the universal minimal set of $\Gamma$. 

If $a, b \in A \otimes \mathcal K$ are positive elements such that $$\mathrm{d}_\tau(a) < \mathrm{d}_\tau(b),\quad \tau \in \tr(A),$$ then $a \precsim b$.
\end{thm}

\begin{proof}
By \cite{RC-5-person}, one may assume that $a , b \in \mathrm{M}_\infty(A)$.

It follows from Proposition \ref{S-URP} and Corollary \ref{S-URP-M} that both dynamical systems have the strong URP of Proposition \ref{S-URP}. The statement of theorem for the universal minimal set $(M, \Gamma)$ follows directly from Proposition \ref{main-thm}.

Let us consider $l^\infty(\Gamma)\rtimes \Gamma$. 
Note that 
\begin{eqnarray*}
l^\infty( \Gamma \times \Int/n\Int) \rtimes (\Gamma \times \Int/n\Int) 
& \cong & (l^\infty(\Gamma) \otimes l^\infty(\Int/n\Int)) \rtimes (\Gamma \times \Int/n\Int)  \\ 
& \cong & (l^\infty(\Gamma) \rtimes \Gamma) \otimes (l^\infty(\Int/n\Int) \rtimes \Int/n\Int) \\
& \cong & \mathrm{M}_n(l^\infty(\Gamma) \rtimes \Gamma). 
\end{eqnarray*}
Then, the strict comparison of positive elements $a, b \in \mathrm{M}_n(l^\infty(\Gamma) \rtimes \Gamma)$ follows from Proposition \ref{main-thm} applied to the dynamical system $(l^\infty(\Gamma \times \Int/n\Int), \Gamma \times \Int/n\Int)$. 
\end{proof}

\begin{rem}
In general,  there exist positive elements $p \in A = l^\infty(\Gamma) \rtimes \Gamma$ and tracial states $\tau_1, \tau_2 \in \tr(A) $ such that $\tau_1(p) = 0$ but $\tau_2(p) \neq 0$. For example, one can take $\Gamma = \Int$ and $p = \chi_{(-\infty, 0]}$. Then it is straightforward to construct two F{\o}lner sequences such that the density of the set $(-\infty, 0]$ along one F{\o}lner sequence is $0$ (so there is a trace $\tau_0$ such that $\tau_0(p) = 0$), while its density along the other F{\o}lner sequence is $1$ (so there is another trace $\tau_1$ such that $\tau_1(p) = 1$). In particular, this implies that $A/J$ in general is not simple, where $J = \{a \in A: \tau(a^*a) = 0,\ \tau \in \tr(A)\}$.

In fact, there is no faithful trace on $A/\mathcal K$ when $\abs{\Gamma} = \infty$, as $l^\infty(\Gamma) / c_0(\Gamma) \subseteq A/\mathcal K$ contains an uncountable family of mutually orthogonal projections. So, none of the invariant means (actually none of the states of $A/\mathcal K$) is faithful on $\beta\Gamma \setminus\Gamma$. Hence, for each invariant mean $\tau$, the invariant ideal 
$$l^\infty(\Gamma)\cap \{a\in A: \tau(a^*a) = 0\}$$ 
is proper, and therefore induces an invariant closed subset of $\beta\Gamma \setminus\Gamma$. (On the other hand, for any non-empty closed invariant subset of  $\Omega' \subseteq \beta\Gamma\setminus\Gamma$, there is always an invariant mean with support inside $\Omega'$.)

Does $A/J$ contain an uncountable family of mutually orthogonal contractions (or even projections)?

%
%Let us call a subset $S \subseteq \Gamma$ sparse if $\tau(\chi_S) = 0$ for all $\tau \in \tr(A)$. Still denoted by $J$ the collection of sparse sets. It is an ideal of the Boolean algebra $\mathcal P(\Gamma) = \{0, 1\}^{\Gamma}$. What the quotient algebra
%$$\mathcal P(\Gamma) / J $$
%looks like?
%
%Note that $$ 0 \oplus \Int $$ is sparse in $\Int^2$. More generally, if a set $S \subseteq \Gamma$ has a property that there are $\gamma_1, \gamma_2, ... \in \Gamma$ such that for every $i \neq j$,
%$$\abs{S\gamma_{i} \cap S\gamma_j} <+\infty, $$ then $S$ is sparse. So $$\{n^2: n\in \Int\} \subseteq \Int$$
%is sparse; indeed, one can take $\gamma_i = i, $ $i=1, 2, ...$; then, if $i \neq j$,  there are at most finitely many pairs $(m, n)$ such that
%$$m^2 + i = n^2 + j,$$
%$$m^2 - n^2 = (m+n)(m-n) = j - i.$$
%
%Is there a concrete family of extreme traces (means)?
%
%Is the C*-algebra $J/\mathcal K$ (separably) stable?

\end{rem}

\subsection{The C*-algebra of the universal minimal set} 
Let us consider the simple C*-algebra $\mathrm{C}(M) \rtimes \Gamma$ (or, more generally, the simple C*-algebra $\mathrm{C}(\Omega) \rtimes \Gamma$, where $(\Omega, \Gamma)$ is a minimal dynamical system which satisfies the strong URP and $\Omega$ is a zero-dimensional compact space, but not necessarily metrizable). Note that the C*-algebra $\mathrm{C}(M) \rtimes \Gamma$ is the unique (non-zero) simple quotient of the Roe algebra $l^\infty(\Gamma) \rtimes \Gamma$: 

Since $(M, \Gamma)$ is minimal and free (\cite{Ellis_1960}), and since $\Gamma$ is amenable, the (full) crossed product C*-algebra $\mathrm{C}(M)\rtimes\Gamma$ is simple (Theorem 3.2 and Remark 3.7 of \cite{Elliott_1980}), and therefore it is isomorphic to the canonical quotient of $l^\infty(\Gamma) \rtimes \Gamma$ by restricting to $M$.

On the other hand, let $B$ be a (non-zero) simple quotient of $l^\infty(\Gamma) \rtimes \Gamma$, and let us show that $B \cong \mathrm{C}(M)\rtimes\Gamma$. Note that $B$ is a crossed product of $l^\infty(\Gamma) \cong \mathrm{C}(\beta\Gamma)$ by $\Gamma$. Write the quotient of $\mathrm{C}(\beta\Gamma)$ in $B$ as $\mathrm{C}(Y)$, where $Y \subseteq \beta\Gamma$ is a closed invariant subset. Then $B$ is a crossed product of $\mathrm{C}(Y)$ by $\Gamma$. Since $B$ is simple, $Y$ must be minimal, as if $Y' \subseteq Y$ is a non-trivial closed invariant subset, then $\mathrm{C}_0(Y \setminus Y') \subseteq B$ would generate a  non-trivial ideal of $B$. Since the action of $\Gamma$ on $\beta\Gamma$ is free (\cite{Ellis_1960}), its restriction to $Y$ is also free. Therefore, the (full) crossed product C*-algebra $\mathrm{C}(Y)\rtimes\Gamma$ is simple (\cite{Elliott_1980}), and it must be isomorphic to $B$. By \cite{Ellis_1960}, all (non-empty) minimal subsets of $(\beta\Gamma, \Gamma)$ are isomorphic (and is isomorphic to the universal minimal set $(M , \Gamma)$), and therefore $B \cong \mathrm{C}(M)\rtimes\Gamma$.

%As pointed out in Corollary 1.3 and Remark 4.3 of \cite{Suzuki-sr1}, this C*-algebra has stable rank one and real rank zero. 

In this subsection, we shall show that the C*-algebra $\mathrm{C}(M) \rtimes\Gamma$ is an increasing union of $\mathcal Z$-absorbing AH algebras with real rank zero, and is approximately divisible. 

%By Corollary \ref{stable-rank-one}, it has stable rank one. In this subsection, let us show that it also has real rank zero. 
First, let us show that the canonical image of $\Kzero(\mathrm{C}(\Omega) \rtimes \Gamma)$ is dense in $\mathrm{Aff}(\mathrm{T}(\mathrm{C}(\Omega) \rtimes \Gamma))$ (with respect to the topology of uniform convergence) if $(\Omega, \Gamma)$ has the strong URP and $\Omega$ is zero dimensional (this actually also holds for the non-simple C*-algebra $l^\infty(\Gamma) \rtimes \Gamma$).

\begin{lem}\label{dense-range}
Let $(\Omega, \Gamma)$ be a topological dynamical system which has the strong URP of Definition \ref{defn-S-URP} (i.e., Rokhlin towers form a partition of unity), and denote by $A$ the crossed product $\mathrm{C}(\Omega) \rtimes \Gamma$. Assume that $\Omega$ is totally disconnected.
%Let $\Gamma$ be a discrete amenable group. 
%Then the C*-algebra  $\mathrm{C}(\Omega) \rtimes \Gamma$ has the following property:
%Let $\Gamma$ be a discrete amenable group, and denote by $A = \mathrm{C}(M) \rtimes \Gamma$, where $(M, \Gamma)$ is the minimal set of $\Gamma$. 
Then the canonical image of $\Kzero(A)$ is dense in $(\mathrm{Aff}(\tr(A)), \norm{\cdot}_\infty)$. Indeed, for any $\rho \in \mathrm{Aff}^+(\tr(A))$ and $\eps>0$, there is a diagonal projection in $\mathrm{M}_\infty(\mathrm{C}(\Omega))$ such that $\abs{\rho(\tau) - \tau(p)} < \eps$, $\tau \in \tr(A)$. %Moreover, the projection $p$ can be chosen to be a diagonal projection of $\mathrm{M}_\infty(\mathrm{C}(\Omega))$.

%In particular, for every positive elements $a, b \in A$ such that $\mathrm{d}_\tau(a) < \mathrm{d}_\tau(b)$ for all $\tau \in \tr(A)$, there is a projection $p \in A$ such that $ \mathrm{d}_\tau(a) < \tau(p) < \mathrm{d}_\tau(b) $ for all $\tau \in \tr(A)$.
\end{lem}

\begin{proof}
Let $\rho: \tr(A) \to \Real $ be a positive continuous affine function, and let $\eps>0$. By Kadison's representation theorem, there is a self-adjoint element $a \in A$ such that $$\rho(\tau) = \tau(a),\quad \tau \in \tr(A).$$

By Proposition \ref{urp} and the assumption that $\Omega$ is totally disconnected, there are a finite dimensional C*-algebra $C \subseteq A$ and a self-adjoint element  $a' \in C$ such that  $$\abs{\tau(a) - \tau(a')} < \eps/2, \quad \tau \in \tr(A).$$ Since $|\Gamma| = \infty$, the C*-algebra $C$ can be chosen such that the order of each minimal direct summand is at least $1/\eps$. Since $\rho$ is positive, one has $\tau(a') > -\eps/2$ for all $\tau\in \tr(A)$. Then, the following compactness argument shows that $C$ can be chosen such that 
\begin{equation}\label{lbd-e}
 \tau(a') > -\eps, \quad \tau \in \tr(C).
 \end{equation}
Otherwise, there would exist a F{\o}lner sequence $(\Gamma_n)$, a sequence $(x_n)  \subseteq \Omega$, finite dimensional C*-algebras $(C_n)$, and self-adjoint elements $a'_n \in C_n$ such that 
$$\tau_{x_n, \Gamma_n}(a'_n) \leq -\eps,$$
where $\tau_{x_n, \Gamma_n}$ is the probability measure
$$\frac{1}{\abs{\Gamma_n}}\sum_{\gamma \in \Gamma_n} \delta_{x_n \gamma}.$$

Note that the probability measures $ \tau_{x_n, \Gamma_n}$, $n=1, 2, ...$, have an accumulation point, say $\tau_\infty$, which is necessarily an invariant probability measure (so it induces a trace of $A$, which will still be denoted by $\tau_\infty$). Then, with $n$ sufficiently large, 
$$\tau_\infty(a) \approx_{\eps/2} \tau_\infty(a'_n) \approx_{\eps/4} \tau_{x_n, \Gamma_n}(a) \leq -\eps,$$
which implies that $\rho(\tau_\infty) = \tau_\infty(a) < 0$, a contradiction to the positivity of $\rho$.

Then, since $C$ is finite dimensional and such that the order of each minimal direct summand is at least $1/\eps$, by \eqref{lbd-e}, there is a projection $p \in C$ such that
$$\abs{\tau(p) - \tau(a')} < \eps, \quad \tau \in \tr(C),$$
and thus, $$ \abs{\tau(p) - \rho(\tau)} < 2\eps, \quad \tau \in \tr(A), $$
as desired.
\end{proof}

It is worth pointing out that the limit of each increasing sequence of strictly positive affine functions on $\mathrm{T}(A)$ arises as the rank function of an open subset of $\Omega$:
\begin{cor}[cf.~\cite{Thiel-sr1}]
Let $(\Omega, \Gamma)$ be a topological dynamical system which has the strong URP of Definition \ref{defn-S-URP} (i.e., Rokhlin towers form a partition of unity), and denote by $A$ the crossed product $\mathrm{C}(\Omega) \rtimes \Gamma$. Assume that $\Omega$ is totally disconnected.

Let $ \rho: \mathrm{T}(A) \to [0, +\infty)$ be a function which is the pointwise limit of an increasing sequence of strictly positive continuous affine function on $\mathrm{T}(A)$, and assume that $\norm{\rho}_\infty \leq 1$. Then there is an increasing sequence of projections $p_1 \leq p_2 \leq \cdots$ in $\mathrm{C}(\Omega)$ such that 
$$ \rho(\tau) = \lim_{n \to\infty}\tau(p_n),\quad \tau \in \mathrm{T}(A). $$ In other words, the affine function $\rho$ can be realized as the rank function of an open set of $\Omega$.
\end{cor}

\begin{proof}
Note there is a sequence of continuous affine functions $\rho_n: \tr(A) \to [0,  +\infty)$ such that $$\rho_n(\tau) < \rho(\tau), \quad \tau \in \tr(A)$$ and $$ \lim_{n \to \infty} \rho_n(\tau) = \rho(\tau),\quad \tau \in \tr(A). $$ Note that, since $\tr(A)$ is compact and $\rho_n$, $n=1, 2, ...$, are continuous, there are $\delta_n>0$, $n=1, 2, ..., $ such that 
$$ \rho_{n+1}(\tau) - \rho_n(\tau) > \delta_n, \quad n=1, 2, ... .$$ Without loss of generality, one may assume that $\delta_n < \delta_{n+1}$, $n=1, 2, ...$.

By Lemma \ref{dense-range}, there are projections $p_i \in \mathrm{C}(\Omega)$, $i=1, 2, ...$, such that
$$\abs{\rho_n(\tau) - \tau(p_n)} < \frac{1}{2}\delta_n,\quad n=1, 2, ... .$$ In particular, $$\tau(p_n) < \tau(p_{n+1}), \quad \tau \in \tr(A).$$
Then, for each $n$, by Proposition \ref{urp} (and the argument of Proposition \ref{cos}), there is a unital sub-C*-algebra $\mathrm{C}(\Omega) \subseteq C \subseteq \mathrm{C}(\Omega) \rtimes \Gamma$ such that
$$C = \bigoplus_{s = 1}^S \mathrm{M}_{\abs{\Gamma_s}}(\mathrm{C}(Z_s)),$$
with $\mathrm{C}(\Omega)$ the diagonal subalgebra and $\mathcal Z_s$, $s=1, ..., S$ zero dimensional, such that
$$\mathrm{rank}(p_n)(x) < \mathrm{rank}(p_{n+1})(x), \quad x \in Z_s,\ s=1, ..., S.$$
Then, by Lemma \ref{matrix-C-0}, there is a unitary $u \in A$ such that
$$u^*p_{n+1}u \in \mathrm{C}(\Omega) \quad \mathrm{and} \quad p_n \leq u^*p_{n+1}u.$$ Thus, upon replacing $p_n$ by a suitable conjugation, one may assume that $$p_1 \leq p_2 \leq \cdots,$$
as desired. 
\end{proof}

\begin{rem}
In the case $\mathrm{C}(\Omega) = l^\infty(\Gamma)$, write $p_n = \chi_{E_n}$ where $E_n \subseteq  \Gamma$. Then $E_1 \subseteq E_2 \subseteq \cdots$, and
$\rho(\tau) = \lim_{n \to\infty}\tau(\chi_{E_n})$. But, in general, $\rho(\tau) \neq \tau(\chi_{\bigcup_{n=1}^\infty E_n})$, as $\tau$ is not a normal state on $l^\infty(\Gamma)$.
\end{rem}

\begin{lem}\label{sep-subalgebra}
Let $(\Omega, \Gamma)$ be a topological dynamical system which has the strong URP of Definition \ref{defn-S-URP} (i.e., there are Rokhlin towers forming  a partition of unity), and assume that $\Omega$ is totally disconnected. Then, for any countable set $\mathcal F \subseteq \mathrm{C}(\Omega)$, there is a separable unital sub-C*-algebra $C=\mathrm{C}(\Omega') \subseteq \mathrm{C}(\Omega)$ which is invariant under the action of $\Gamma$ such that
\begin{enumerate}
\item $\mathcal F \subseteq \mathrm{C}(\Omega')$,
\item $\Omega'$ is totally disconnected,
\item $(\Omega', \Gamma)$ has the strong URP.
%\item $(\Omega', \Gamma)$ is minimal.
\end{enumerate}
\end{lem}

\begin{proof}
The statement follows from the Blackadar argument.

Let $K_1 \subseteq K_2 \subseteq \cdots \subseteq \Gamma$ be a sequence of finite set such that $\bigcup_{i=1}^\infty K_i = \Gamma$, and let $\eps_1 > \eps_2 > \cdots $ be a sequence of positive numbers such that $\lim_{i\to\infty} \eps_i = 0$. Since $(\Omega, \Gamma)$ has the URP with Rokhlin towers forming a partition of unity, for each $i=1, 2, ...$, there are projection $p_{i, s} \in \mathrm{C}(\Omega)$, and finite sets $\Gamma_{i, s} \subseteq \Gamma$, $s = 1, ..., S_i$, such that
\begin{enumerate}
\item $\Gamma_{i, s}$, $s=1, ..., S_i$, are $(K_i, \eps_i)$-invariant, and
\item $\sum_{s=1}^{S_i} \sum_{\gamma \in \Gamma_{i, s}} \gamma(p_{i, s}) = 1_{\mathrm{C}(\Omega)}$. 
\end{enumerate}

Set $$C_0:=\textrm{C*}\{\gamma(p_{i, s}), \gamma(f): s=1, ..., S_i,\ i=1, 2, ..., f \in \mathcal F,\ \gamma \in \Gamma\} \subseteq \mathrm{C}(\Omega).$$ It is a separable sub-C*-algebra which is invariant under $\Gamma$ and contains $\{p_{i, s}: s=1, ..., S_i,\ i=1, 2, ...\}$ and $\mathcal F$.

Pick a countable set $\mathcal G_0 \subseteq C_0$ which is dense in the set of  self-adjoint elements of $C_0$. Since $\mathrm{C}(\Omega)$ has real rank zero, there is a countable set $\mathcal H_0 \subseteq \mathrm{C}(\Omega)$ consisting of invertible self-adjoint elements such that $\mathcal G_0 \subseteq \overline{\mathcal H_0}$. Consider the separable C*-algebra
$$C_1: = \textrm{C*}\{C_0, \gamma(\mathcal H_0): \gamma \in \Gamma\} \subseteq \mathrm{C}(\Omega).$$
Then the C*-algebra $C_1$ is invariant under the action of $\Gamma$, $C_0 \subseteq C_1$, and any self-adjoint element of $C_0$ can be approximated by invertible self-adjoint elements of $C_1$.

Repeating this process, one obtains a sequence of $\Gamma$-invariant separable unital sub-C*-algebras $$C_0 \subseteq C_1 \subseteq \cdots \subseteq \mathrm{C}(\Omega)$$
such that each self-adjoint element of $C_n$, $n=0, 1, ...$, can be approximated by invertible self-adjoint elements of $C_{n+1}$. Then the C*-algebra $$C:=\overline{\bigcup_{n=0}^\infty C_n}$$ is $\Gamma$-invariant, separable, and has real rank zero. 

It is clear that $\mathcal F \subseteq C$. Since for each $i=1, 2, ...$, $$\{p_{i, s}: s=1, ..., S_i\} \subseteq C,$$ the dynamical system $(C, \Gamma)$ has partitions by Rokhlin towers which are $(K_i, \eps_i)$-invariant, which is the strong URP.
\end{proof}

Now, let us turn our attention to minimal systems. Note that the quotient system $(\Omega', \Gamma)$ obtained above is minimal if $(\Omega, \Gamma)$ is minimal. Also recall from \cite{BKR-ADiv} that a unital C*-algebra $A$ is approximately divisible if for any finite set $\mathcal F \subseteq A$ and any $\eps>0$, there is a unital embedding $\phi: F \to A$, where $F=\mathrm{M}_2(\Comp) \oplus \mathrm{M}_3(\Comp)$, such that $$ \norm{a \phi(x) - \phi(x) a} < \eps,\quad a \in \mathcal F,\ x \in F,\ \norm{x} \leq 1.$$

\begin{thm}\label{min-alg}
Let $(\Omega, \Gamma)$ be a minimal topological dynamical system which has the strong URP of Definition \ref{defn-S-URP} (i.e., there are Rokhlin towers forming a partition of unity), and assume that $\Omega$ is totally disconnected. 

Then, for any countable set $\mathcal F \subseteq \mathrm{C}(\Omega) \rtimes \Gamma $, there is a unital %simple separable nuclear $\mathcal Z$-absorbing 
sub-C*-algebra $B \subseteq \mathrm{C}(\Omega) \rtimes \Gamma$ which is isomorphic to a simple unital $\mathcal Z$-absorbing AH algebra with real rank zero
%satisfies the UCT and has real rank zero 
such that $$\mathcal F \subseteq B.$$               
In particular, the C*-algebra $\mathrm{C}(\Omega) \rtimes \Gamma$ is approximately divisible.
\end{thm}

\begin{proof}
Approximating each element of $\mathcal F$ with elements in the algebraic crossed  product, one may assume that each element of $\mathcal F$ is a $\mathrm{C}(\Omega)$-valued function on $\Gamma$ with finite support. Then, denote by  $\mathcal F' \subseteq \mathrm{C}(\Omega)$ the countable set of all the coefficients of elements of $\mathcal F$ (i.e., their values as functions $\Gamma \to \mathrm{C}(\Omega)$). By Lemma \ref{sep-subalgebra}, there is a $\Gamma$-invariant separable sub-C*-algebra $\mathrm{C}(\Omega') \subseteq \mathrm{C}(\Omega)$ such that
$\mathcal F' \subseteq \mathrm{C}(\Omega'). $

Regarding $B:=\mathrm{C}(\Omega') \rtimes \Gamma$ as a sub-C*-algebra of $\mathrm{C}(\Omega) \rtimes \Gamma$, one then has $\mathcal F \subseteq B$. The C*-algebra $B$ is simple since $(\Omega', \Gamma)$ is minimal and has the strong URP (hence $(\mathrm{C}(\Omega'), \Gamma)$ is topologically free; see Remark \ref{tp-free}). It is also nuclear and satisfies the UCT since $\Gamma$ is amenable. 

Since $\Omega'$ is zero dimensional, the system $(\Omega', \Gamma)$ has zero mean dimension. It then follows from Theorem 4.8 of  \cite{Niu-MD-Z-absorbing} that $B \cong B \otimes \mathcal Z$. (The freeness assumption of Theorem 4.8 of  \cite{Niu-MD-Z-absorbing} actually is not necessary. This assumption is not used in the proof of Proposition 3.1 and Lemma 4.1 of \cite{Niu-MD-Z-absorbing}. When the freeness is used in Corollary 3.2 and Theorem 4.8, it is strict comparison that is needed. But, by Proposition \ref{main-thm}, strict comparison actually just follows from the URP.)

%By Theorem \ref{main-thm}, the C*-algebra $B$ has the strict comparison property. Together with Proposition 4.7 of \cite{Niu-MD-Z-absorbing}, this implies that the C*-algebra $B$ is tracially $\mathcal Z$-absorbing (the freeness assumption in Lemma 4.1 and Proposition 4.7 of \cite{Niu-MD-Z-absorbing} indeed are not necessary for the proofs). Since $B$ is nuclear, one has that $B$ is $\mathcal Z$-absorbing. 

Since $B$ is finite, it has stable rank one (\cite{Ror-Z-stable}). By Lemma \ref{dense-range}, the image of $\Kzero(B)$ is uniformly dense in $\mathrm{Aff}(\mathrm{T}(B))$, and hence, by Theorem 7.2 of \cite{Ror-Z-stable}, the C*-algebra $B$ has real rank zero.

Therefore, $B$ is a unital simple separable nuclear $\mathcal Z$-absorbing C*-algebra which satisfies the UCT. It is classified by the conventional Elliott invariant. Since $B$ has real rank zero, it is isomorphic to an AH algebra (without dimension growth), and therefore is approximately divisible (\cite{EGL-ADiv}). 
\end{proof}

\begin{cor}\label{AH-structure}
Let $(\Omega, \Gamma)$ be a minimal topological dynamical system which has the strong URP of Definition \ref{defn-S-URP} (i.e., there are Rokhlin towers forming a partition of unity), and assume that $\Omega$ is totally disconnected. 

Then there is an increasing net $A_\lambda$, $\lambda \in \Lambda$, of sub-C*-algebras of $A: = \mathrm{C}(\Omega) \rtimes \Gamma$ such that 
each $A_\lambda$, $\lambda \in \Lambda$, is a  (separable) simple $\mathcal Z$-absorbing AH algebra with real rank zero, and $$\bigcup_{\lambda \in \Lambda} A_\lambda  = A. $$
\end{cor}
\begin{proof}
Denote by $\Lambda$ set of all finite subsets of $A$. It is an upward directed set under the inclusion. 

Put $A_{\O} = 1_A$.

For each $a \in A$, by Theorem \ref{min-alg}, there is a (separable) $\mathcal Z$-absorbing real rank zero AH algebra $A_a \subseteq A$ such that $a \in A_a$. Then, for each $\lambda \in \Lambda$ such that $\abs{\lambda} = 1$, define $A_\lambda = A_a$, where $\lambda = \{a\}$.

Assume that for some $n\in \mathbb N$, the (separable) C*-algebras $A_\lambda$ are define for all $\abs{\lambda} \leq n$ with the following properties: 
\begin{enumerate}
\item[(1)] $\lambda \subseteq A_\lambda$ and
\item[(2)] $A_{\lambda_1} \subseteq A_{\lambda_2}$ if $\lambda_1 \subseteq \lambda_2$.
\end{enumerate}

For each $\lambda \in \Lambda$ with $\abs{\lambda} = n+1$, consider the separable sub-C*-algebras $$A_{\lambda'},\quad \lambda' \subseteq \lambda,\ \abs{\lambda'} = n. $$
Note that there are only finitely many such $A_{\lambda'}$. By Theorem \ref{min-alg}, there is a (separable) $\mathcal Z$-absorbing real rank zero AH algebra $A_\lambda \subseteq A$ such that 
$$A_{\lambda'} \subseteq A_\lambda, \quad \lambda' \subseteq \lambda,\ \abs{\lambda'} = n.$$ 

Construct such $A_\lambda$ for all $\lambda \in \Lambda$ that $\abs{\lambda} = n+1$. Then, the collection of sub-C*-algebras $$A_\lambda, \quad \lambda \in \Lambda,\ \abs{\lambda} \leq n+1$$ satisfy 
\begin{enumerate}
\item[(1')] $\lambda \subseteq A_\lambda$ and
\item[(2')] $A_{\lambda_1} \subseteq A_{\lambda_2}$ if $\lambda_1 \subseteq \lambda_2$.
\end{enumerate}

Repeating this process, one obtains a net of (separable) $\mathcal Z$-absorbing real rank zero AH algebras $A_\lambda \subseteq A$, $\lambda \in \Lambda$, such that
\begin{enumerate}
\item[(1'')] $\lambda \subseteq A_\lambda$ and
\item[(2'')] $A_{\lambda_1} \subseteq A_{\lambda_2}$ if $\lambda_1 \subseteq \lambda_2$,
\end{enumerate}
which has the desired properties.
\end{proof}

\begin{rem}

If $(\Omega, \Gamma)$ is not minimal, the crossed product C*-algebra might fail to have stable rank one; see, for example, \cite{Poon-PAMS-89} and \cite{BNS-Cantor} for Cantor systems of $\Int$-actions. The uniform Roe algebra $A = l^\infty(\Gamma) \rtimes \Gamma$ also provides such examples: Assume $\Int \subseteq\Gamma$, and consider the projection $$p = \chi_{(-\infty, 0]} \in l^\infty(\Gamma). $$ To see the stable rank of $A$ is not $1$, it is enough to show that the stable rank of  the hereditary subalgebra $pAp$ is not $1$. Denote by $u$ the canonical unitary corresponding to $1 \in \Int$. Then the element 
$$ v := p u p$$
satisfies
$$vv^* = p(upu^*)p = p, \quad\mathrm{and} \quad v^*v = p(u^*pu)p =  u^*pu = \chi_{(-\infty, -1]}. $$ So, $v$ is an isometry. Its image in the quotient $pAp/(pAp\cap \mathcal K)$, where $\mathcal K = c_0(\Gamma) \rtimes \Gamma$,  is a unitary with non-zero index 
$$[p - vv^*]_0 - [ p - v^*v]_0 = -[\chi_{\{-1\}}]_0 \in \Kzero(\mathcal K) \cong \Int.$$
Thus, the stable rank of $pAp$ is not $1$. (Indeed, if $\Gamma  = \Int^d$, it is shown in \cite{LW-Roe} that the stable rank of $A$ is $2$.) (It is also shown in \cite{LW-Roe} that the real rank of $l^\infty(\Gamma) \rtimes \Gamma$ in general fails to be zero due to non-zero exponential maps.)
\end{rem}

\begin{rem}
What are the K-groups and the trace simplex of $\mathrm{C}(M) \rtimes \Gamma$? Note that, since $\mathrm{C}(M) \rtimes \Gamma$ has real rank zero, its trace simplex is canonically isomorphic to the state space of its order-unit $\Kzero$-group.
\end{rem}

\begin{rem}
Let $\Gamma_1$ and $\Gamma_2$ be discrete amenable groups, and consider their universal minimal sets $(M_1, \Gamma_1)$ and $(M_2, \Gamma_2)$. When are these dynamical systems topologically orbit equivalent (\cite{GPS-Cantor})? 
\end{rem}

\bibliographystyle{plainurl}
\bibliography{operator_algebras}

\end{document}